\documentclass[preprint,12pt]{elsarticle}

\usepackage{amsmath}
\usepackage{amssymb}
\usepackage{amsthm}
\usepackage{mathrsfs}
\usepackage{latexsym}
\usepackage{mathtools}
\usepackage{enumerate}
\usepackage{multirow}
\usepackage{longtable}
\usepackage{rotating}
\usepackage{lscape}
\usepackage[top=2cm, bottom=2cm, left=2cm, right=1.0cm]{geometry}

\usepackage[bookmarks,pagebackref]{hyperref}

\numberwithin{equation}{section}
\usepackage{url}
\usepackage{graphicx}
\usepackage{color}

\newtheorem{thm}{Theorem}
\newtheorem{lem}{Lemma}[section]
\newtheorem{prop}[lem]{Proposition}
\newtheorem{cor}[lem]{Corollary}

\newtheorem{defn}[lem]{Definition}

\newdefinition{rem}[lem]{Remark}
\newdefinition{ex}[lem]{Example}
\newproof{pot1}{Proof of Theorem \ref{Thm1}}
\newproof{pot2}{Proof of Theorem \ref{Thm2}}
\newproof{pot3}{Proof of Theorem \ref{Thm3}}

\journal{Topology and its Applications}

\newcommand{\A}{\mathcal{A}}
\newcommand{\F}{\mathcal{F}}
\newcommand{\G}{\mathcal{G}}
\newcommand{\s}{\operatorname{span}}

\newcommand{\R}{\mathbb{R}}

\newcommand{\g}{\mathfrak{g}}
\newcommand{\X}{\mathfrak{X}}

\newcommand{\ad}{\operatorname{ad}}  
\newcommand{\Rank}{\operatorname{rank}}
\newcommand\Lie{\mathrm{Lie}}

\begin{document}

\begin{frontmatter}



\title{Foliations formed by generic coadjoint orbits of Lie groups corresponding to a class 7-dimensional solvable Lie algebras}


\author[1,2]{Tuyen T. M. Nguyen}
\ead{ntmtuyen@dthu.edu.vn}
\address[1]{Faculty of Mathematics and Computer Science, University of Science, VNU Ho Chi Minh City, Vietnam}
\address[2]{Faculty of Mathematics and Computer Science Teacher Education, Dong Thap University, Dong Thap Province, Vietnam}

\author[3]{Vu A. Le\corref{cor1}}
\ead{vula@uel.edu.vn}
\address[3]{Department of Economic Mathematics, University of Economics and Law, 
Vietnam National University -- Ho Chi Minh City, Vietnam}
\cortext[cor1]{Corresponding author.}

\author[4]{Tuan A. Nguyen}
\ead{tuannguyenanh@hcmue.edu.vn}
\address[4]{Department of Primary Education, Ho Chi Minh City University of Education, Vietnam}

\begin{abstract}
	We consider all connected and simply connected 7-dimensional Lie groups 
	whose Lie algebras have nilradical $\g_{5,2} = \s \{X_1, X_2, X_3, X_4, X_5 \colon [X_1, X_2] = X_4, [X_1, X_3] = X_5\}$ of Dixmier. 
	First, we give a geometric description of the maximal-dimensional orbits in the coadjoint representation of all considered Lie groups. 
	Next, we prove that, for each considered group, the family of the generic coadjoint orbits 
	forms a measurable foliation in the sense of Connes. Finally, the topological classification of all these foliations is also provided.
\end{abstract}


\end{frontmatter}


\section{Introduction}\label{sec1}

The study of foliations on manifolds has a long history in mathematics. Since the works of Ehresmann and Reeb \cite{e-r} in 1944, Ehresmann \cite{ehr} in 1951, and Reeb \cite{reeb} in 1952, the foliations on manifolds have enjoyed rapid development. Nowadays, they become the focus of a great deal of research activity (see \cite{law}).
In general, the leaf space of a foliation with quotient topology is a fairly intractable topological space. To improve upon the shortcoming, Connes \cite{con} proposed the notion of measurable foliations in 1982 and associated each such foliation $(V, \F)$ with a C*-algebra $C^*(V, \F)$. It represents the leaf space $V/\F$ in the following sense: when the foliation comes from a fibration (with connected fibers) $p: V \to M$ then $C^*(V, \F)$ is isomorphic to $C_0 (M) \otimes \mathcal{K}$, where $C_0(M)$ is the algebra of continuous complex-valued functions defined on $M$ vanishing at infinity and $\mathcal{K}$ denotes the C*-algebra of compact operators on an (infinite-dimensional) separable Hilbert space. During the last few decades, these concepts of Connes have become important tools of non-commutative differential geometry and have attracted much attention from mathematicians around the world.

In 1962, Kirillov \cite[Section 15]{kir} invented the method of orbits that quickly became the most important one in the
representation theory of Lie groups and Lie algebras. The key to Kirillov's method of orbits is generic (in a certain sense) orbits in the coadjoint representation (K-orbits, for short) of Lie groups. Hence, the problem of describing the geometry of (generic) K-orbits of each Lie group is very important to study. A concept closely related to Kirillov's method of orbits
is the so-called \emph{MD-group} proposed by Diep \cite{diep} in 1980. An MD-group, in terms of Diep \cite[Section 4.1]{diep}, is a finite-dimensional solvable real Lie group whose K-orbits are of dimension zero or maximal dimension. It is very remarkable that the family of maximal-dimensional K-orbits of an MD-group $G$ forms a so-called measurable foliation in the sense of Connes \cite{con} which is called MD-foliation associated with $G$.

Thus, the nature of the problem concerning MD-groups, MD-foliations, and their Connes’ C*-algebras
is an interesting combination of Kirillov’s method of orbits with Connes’ method in non-commutative geometry. Therefore, this problem is worth studying.
During the period from 1990 to now, there have been much research works on MD-groups and MD-foliations, typically working \cite{vu90, Vu90, vu93}, \cite{v-s} and \cite{v-t,v-h,v-h-t}, ... Up today, the problem of classifying MD-groups and MD-foliations still remains~open. 

Naturally, a problem arises as follows: For a solvable Lie group that is not an MD-group, does the family of its maximal-dimensional K-orbits 
have the same properties as the K-orbits of MD-groups? We hope that the answer is positive. Namely, we hope that some nice properties of MD-groups can be generalized for a larger class of these solvable Lie groups.

According to Dixmier \cite[Proposition 1]{dix}, the class of 5-dimensional real or complex nilpotent Lie algebras consists of
nine algebras which are denoted by $(\g_1)^5$, $(\g_1)^2 \oplus \g_3$, $\g_1 \oplus \g_4$, $\g_{5,1}$, $\g_{5,2}$, 
$\g_{5,3}$, $\g_{5,4}$, $\g_{5,5}$, $\g_{5,6}$. Recently, we have classified 7-dimensional indecomposable solvable Lie algebras 
with nilradical $(\g_1)^2 \oplus \g_3, \g_1 \oplus \g_4, \g_{5,2}, \g_{5,4}$ in \cite{v-t-t-t-t}. Combining these results with those of 
Ndogmo and Winternitz \cite{n-w}, Rubin and Winternitz \cite{r-w}, \v Snobl and Kar\'asek \cite{s-k}, \v Snobl and Winternitz \cite{s-w05,s-w09},
Gong \cite{gong}, Parry \cite{par}, Hindeleh and Thompson \cite{h-t},
we achieve a full classification of 7-dimensional indecomposable solvable Lie algebras. 
As a continuation of \cite{v-t-t-t-t}, we combine the idea of Kirillov's method of orbits with Connes' method in non-commutative geometry
to consider all Lie groups corresponding to Lie algebras classified in \cite{v-t-t-t-t}. As stated above, we hope that some nice properties of MD-groups can be generalized for a larger class of these solvable Lie groups. 
Namely, we started to study Lie groups corresponding to Lie algebras in Table \ref{tab1} which are not too complicated 
but also not trivial among the groups classified in \cite{v-t-t-t-t}.
In fact, we have solved quite completely in \cite{t-v} the similar problem of MD-groups for five families of 
Lie groups corresponding to Lie algebras $\G_2$, $\G_3$, $\G^{00}_4$, $\G_9$, $\G^\lambda_{10}$ in Table \ref{tab1}.

\begin{longtable}{|p{.06\textwidth} |p{.500\textwidth} |p{.3\textwidth}|}
	\caption{7-dimensional solvable Lie algebras with\\
		\hskip3cm nilradical $\mathfrak{g}_{5,2}$}\label{tab1} \\
	\hline No. & Triplets $(a_X, a_Y, [X,Y])$ & Conditions \endfirsthead
	\caption*{\textbf{Table 1} (continued)} \\
	\hline No. & $(\ad_X, \ad_Y, [X,Y])$ & Conditions \\
	\hline \endhead \hline
	$\G_1^\lambda$ & $(1, -1, 0, 0, 1)$, $(0, 0, 1, 0, 1)$, $\lambda X_4$ & $\lambda\in\{0,1\}$ \\
	$\G_2$ & $(1, 0, 0, 1, 1)$, $(0, 0, 1, 0, 1)$ &\\
	$\G_3$ & $(0, 1, 0, 1, 0)$, $(0, 0, 1, 0, 1)$& \\
	$\G^{\lambda_1\lambda_2}_4$ & $(1, 0, \lambda_1, 1, 1+\lambda_1)$, $(0, 1, \lambda_2, 1, \lambda_2)$ 
	& $(\lambda_{1},\lambda_2)\neq (-1,0)$, $\lambda_1+1\neq\lambda_2$\\
	$\G_5$ & $(0, 0, 1, 0, 1)$, $(1, 1, 0, 2, 1) + E_{12}$& \\
	$\G^{\lambda}_6$ & $(1, 1, \lambda, 2, 1+\lambda)$, $(0, 0, 1, 0, 1) + E_{12}$ & $\lambda\in \R$\\
	$\G_7$ & $(0, 1, 1, 1, 1)$, $(1, 1, 0, 2, 1) + E_{25}$& \\
	$\G^{\lambda}_8$ & $(1, 1+\lambda, \lambda, 2+\lambda, 1+\lambda)$, $(0, 1, 1, 1, 1) + E_{25}$& $\lambda\in \R$\\
	$\G_9$ & $(0, 0, 1, 0, 1)$, $(0, 1, 0 , 1, 0) + E_{35}$& \\
	$\G^{\lambda}_{10}$ & $(0 , 1 , \lambda, 1, \lambda)$, $(0, 0, 1, 0, 1) + E_{35}$ & $\lambda\in \R$ \\
	$\G_{11}$ & $(0, 1, 1, 1, 1)$, $(1, 0, 0, 1, 1) + E_{23} + E_{45}$& \\
	$\G^{\lambda}_{12}$ & $(1, \lambda, \lambda, 1+\lambda, 1+\lambda)$, $(0, 1, 1, 1, 1) + E_{23} + E_{45}$& $\lambda \in \R \setminus \{-1\}$\\ 
	$\G_{13}^{\lambda}$ & $(0, 1, 1, 1, 1)$, $(\lambda, S_{01}, S_{\lambda1})$ & $\lambda \geq 0$ \\
	$\G_{14}^{\lambda_{1}\lambda_2}$ & $(1, \lambda_1, \lambda_1, 1+\lambda_1, 1+\lambda_1)$, $(0, S_{\lambda_21}, S_{\lambda_21})$ & $\lambda_2 \geq 0, \lambda_1\neq-1$ \\ \hline
	$\G_{15}$ & $(0, S_{01}, S_{01})$, $(0, 1, 1, 1, 1) +E_{25} - E_{34}$& \\ 
	$\G_{16}^{\lambda}$ & $(0, S_{01}, S_{01}) + E_{25}$, $(0, 1, 1, 1, 1) + \lambda(E_{25} - E_{34})$ & $\lambda \geq 0$ \\ \hline
\end{longtable}

Here, as usual, we use the convention that $[X,Y]$ disappears if $[X,Y]=0$; 
we denote by $(a_1, \dots, a_5)$ the diagonal $(5 \times 5)$-matrix diag$(a_1, \dots, a_5)$, $E_{ij}$ is the $(5 \times 5)$-matrix whose only non-zero entry is 1 in row $i$, column $j$ ($1 \leq i, j \leq 5$) and 
$S_{ab}: = \begin{bmatrix}
a&b\\ -b&a	
\end{bmatrix}$.

In this paper, we continue to study the similar problem of MD-groups for all Lie groups corresponding to the remaining Lie algebras in Table \ref{tab1}. 
The main results of the paper are as follows. First, we describe the geometric pictures of maximal-dimensional K-orbits of the considered Lie groups. 
Next, we prove that the families of all generic maximal-dimensional K-orbits of the Lie groups corresponding to considered Lie algebras in Table \ref{tab1}
form measurable foliations (in the sense of Connes \cite{con}). Finally, we give the topological classification of these foliations.

The paper is organized into four sections, including this introduction. Section \ref{sec2} describes the geometry of generic K-orbits of considered Lie groups.
Section \ref{sec3} is devoted to dealing with foliations formed by generic K-orbits.
We also give some concluding remarks in the last section.

\section{Generic K-orbits of considered Lie groups}\label{sec2}

In \cite[Section 2]{t-v}, we recalled some notions and well-known results about adjoint and coadjoint representations, 
foliations and measurable foliations which will be used later. For more details, we refer to the book of Kirillov \cite[$\S$15 and $\S$6]{kir}. 
Moreover, by \cite{v-t-t-t-t}, we have sixteen families of connected and simply connected (real solvable) Lie groups corresponding 
to the indecomposable Lie algebras given in Table \ref{tab1}. Note that throughout this section the condition for the parameters 
$\lambda, \lambda_1, \lambda_2$ in each Lie group G is the same as the condition for the parameters in each algebra $\G = \Lie(G)$ 
in Table \ref{tab1} if nothing more is said. In this section, we shall continue using the notations as in \cite{t-v}. 
Recall that the picture of maximal-dimensional K-orbits of Lie groups $G \in \{G_2, G_3, G_4^{00}, G_{9}, G_{10}^\lambda\}$ 
described in \cite[Theorem 19]{t-v}. Here, we will describe the picture of maximal-dimensional K-orbits of the remaining Lie groups.

For convenience, we put $\A$ 
be sixteen connected and simply connected Lie groups corresponding to sixteen Lie algebras listed in Table \ref{tab1}.
For $G \in \A$, denote by $\G^*$ the dual space of the Lie algebra $\G = \Lie(G)$ of $ G $. Clearly, we can identify $\G^* \equiv \R^7$ 
by fixing in it the basis $(X_1^*, \dots, X_5^*, X^*, Y^*)$ which is the dual of the basis $(X_1, \dots, X_5, X, Y)$ of $ \mathcal G $. Let $ F = \alpha_1X_1^*+\dots+\alpha_5X_5^*+\alpha X^*+\beta Y^* \equiv (\alpha_1,\dots,\alpha_5,\alpha,\beta) $
be an arbitrary element of $\G^* \equiv \R^7$. The notation $\Omega_F$ will be used to denote the K-orbit of $G$
containing $F$. By \cite[Proposition 17, Corollary 7]{t-v}, $G$ is exponential (then $\Omega_F=\Omega_F(\G)$) 
if it belongs to $\A_1: = \big\{G^\lambda_1$, $G_2$, $G_3$, $G_4^{\lambda_1\lambda_2}$, $G_5$, $G^\lambda_6$,
$G_7$, $G^\lambda_8$, $G_9$, $G^\lambda_{10}$, $G_{11}$, $G_{12}^\lambda \big\}$, and not exponential otherwise: $G \in \A_2: = \left\lbrace G_{13}^\lambda, G_{14}^{\lambda_1\lambda_2},
G_{15}, G_{16}^\lambda \right\rbrace$. Of course, $\A = \A_1 \sqcup \A_2$. Besides, we have the following proposition.

\begin{prop}[see \cite{kir}]\label{Partition}
	Assume that $G$ is connected. Then $\Omega_F(\G) = \Omega_F$ for all $F \in \G^*$ if the family $\left\lbrace \Omega_F(\G) \right\rbrace_{F \in \G^*}$ 
	forms a partition of $\G^*$ and all $\Omega_{F'}(\G), F' \in \Omega_F$ are either closed or open (relatively) 
	in $\Omega_F, F \in \G^*$. 
\end{prop}

Therefore, we also have $\Omega_F = \Omega_F(\G)$ for $G \in \A_2$. 
That means $\Omega_F = \Omega_F(\G)$ for all $G \in \A$. 

The geometric picture of the K-orbits of considered Lie groups is given as follows.

\begin{thm}\label{picture K-orbit}
	Assume that $G \in \A$. Denote by   $(\alpha_1,\dots,\alpha_5,\alpha,\beta)$ the coordinate of $F \in \G^* \equiv\R^7$
	with respect to the basis $(X_1^*, \dots, X_5^*, X^*, Y^*)$ which is dual one of the fixed basis $(X_1, \dots, X_5, X, Y)$ of $\G$. 
	Then, the maximal dimension of K-orbits of $G$ is exactly six, and the picture of six-dimensional K-orbits of $G$ is given in Table \ref{tab2}.
\end{thm}

\begin{proof}
	The proof will take place through three steps as follows.
	\begin{itemize}
	\item Step 1: For every pair $X, Y \in \G$ and $F(\alpha_1,\dots,\alpha_5,\alpha,\beta) \in \G^*$,  we need to determine the matrix of the bilinear form $B_F(X, Y) := \langle F, [X, Y]\rangle$.
	\item Step 2: Find necessary and sufficient conditions for $\Rank B_F=6$ (the maximal dimension of K-orbits of $G$ is exactly six).
	\item Step 3: Describe explicitly every six-dimensional K-orbits of $G$, i.e. describe $\Omega_F \equiv \Omega_F(\G) = \{F_U \, | \, U \in \G\} \subset \G^* \equiv \R^7$  where
	\[
	F_U = \sum_{i=1}^{5}x^*_iX^*_i+x^*X^*+y^*Y^* 	\equiv (x^*_1 , x^*_2, x^*_3, x^*_4, x^*_5, x^*,y^*) \in \G^*
	\]
	is 
	defined by $\left\langle F_U, T\right\rangle =\left\langle F,\exp(\ad_U) T\right\rangle $ for all $T \in \G.$
\end{itemize}	
	
	First of all, the assertion of Theorem \ref{picture K-orbit} for $G \in \left\lbrace G_2, G_3, G_9, G^\lambda_{10} \right\rbrace$ 
	had been proved in \cite{t-v}. Below we will give detailed calculations for three groups $G_4^{\lambda_{1}\lambda_2}$, $G_{12}^{\lambda}$ and $G_{13}^{\lambda}$
	which do not belong to the family of groups considered in \cite{t-v}.
	
	Let $G = G^{\lambda_1\lambda_2}_4$. In Step 1, by direct computations, the matrix of the bilinear form $B_F$ is as follows
	\[
		B_F = \begin{bmatrix}
			0 & \alpha_4 & \alpha_5 & 0 & 0 & - \alpha_1 & 0 \\
			- \alpha_4 & 0 & 0 & 0 & 0 &0 & -\alpha_2 \\
			- \alpha_5 & 0 & 0 & 0 & 0 & -\lambda_1\alpha_3 & -\lambda_2 \alpha_3 \\
			0 & 0 & 0 & 0 & 0 & - \alpha_4 & -\alpha_4\\
			0 & 0 & 0 & 0 & 0 & -(1+\lambda_1) \alpha_5 & -\lambda_2 \alpha_5\\
			\alpha_1 & 0 & \lambda_1\alpha_3 & \alpha_4 & (1+\lambda_1)\alpha_5 & 0 & 0\\
			0 & \alpha_2 & \lambda_2\alpha_3 & \alpha_4 & \lambda_2\alpha_5 & 0 & 0
		\end{bmatrix}.
	\]
	For Step 2, it is obvious that the maximal rank of $B_F$ is six. Moreover
	\[
		\begin{array}{l l l}
			\Rank (B_F) = 6 & \Leftrightarrow & \alpha_4 = 0 \neq \alpha_2\alpha_5 \quad \text{ or } \quad \alpha_4 \neq 0 \neq \alpha_3^2+\alpha_5^2.
		\end{array}
	\]
	In Step 3, we will describe the picture of maximal-dimensional K-orbits $\Omega_F(\G) = \{F_U \, | \, U \in \G\} \subset \G^* \equiv \R^7$ of $G$. 
To determine $F_U$ with $U \equiv (x_1,x_2,x_3,x_4,x_5,x,y) \in \G$, we use MAPLE to determine $\exp(\ad_U)$ 
with respect to the basis $(X_1, X_2, X_3, X_4, X_5, X, Y)$. We can find that
\[
\exp(\ad_U) =
\begin{bmatrix}
e^{x} & 0 & 0 & 0 & 0 & -x_1\xi & 0 \\
0 & e^{y} & 0 & 0 & 0 & 0  & x_2\varepsilon \\
0 & 0 & e^{\lambda_1x+\lambda_2y} & 0 & 0 & \lambda_1x_3\zeta & \lambda_2x_3\zeta \\
x_2e^{x}\varepsilon	 & x_1e^{y}\xi & 0 & e^{x+y} &0 & b_4 & d_4 \\
a_4 & 0 & x_1e^{x\lambda_1+y\lambda_2}\xi& 0 & e^{\lambda_1x+\lambda_2y+x} & c_4 & f_4 \\
0 & 0 & 0 & 0 & 0 & 1 & 0	\\
0 & 0 & 0 & 0 & 0 & 0 & 1
\end{bmatrix},
\]
where $\xi= \frac{e^{x}-1}{x}, \quad \varepsilon=\frac{1-e^{y}}{y}, \quad \zeta= \frac{1-e^{x\lambda_1+y\lambda_2}}{x\lambda_1+y\lambda_2}$ and $a_4, b_4, c_4, d_4, f_4$ are expressions containing parameters $x_1, x_3, x_5, x, y$.
By a direct computation, we get
\[\begin{array}{l}
x^*_1=\alpha_1 e^{x}+\alpha_4 \frac{x_2e^{x}(1-e^{y})}{y} + \alpha_5\frac{x_3e^{x}(1-e^{\lambda_1x+\lambda_2y})}{\lambda_1x+\lambda_2y}
\\
x^*_2=\alpha_2e^{y}+ \alpha_4\frac{x_1e^{y}(e^{x}-1)}{x}   
\\
x^*_3=  \alpha_3 e^{\lambda_1x+\lambda_2y} +\alpha_5\frac{x_1e^{\lambda_1x+\lambda_2y}(e^{x}-1)}{x} 
\\
x^*_4= \alpha_4e^{x+y}
\\
x^*_5= \alpha_5 e^{\lambda_1x+x+\lambda_2y}
\\
x^*=  \alpha_1\frac{x_1(1-e^{x})}{x} +\alpha_3\lambda_{1}x_3\frac{1-e^{\lambda_1x+\lambda_2y}}{\lambda_1x+\lambda_2y}  +\alpha_4p + \alpha_5 q+ \alpha 
\\
y^*=\alpha_2x_2\frac{1-e^{y}}{y}+ \alpha_3\lambda_2x_3\frac{1-e^{\lambda_1x+\lambda_2y}}{\lambda_1x+\lambda_2y} +\alpha_4r+\alpha_5 s+\beta.
\end{array}\]
In order to describe all maximal-dimensional K-orbits of $G = G^{\lambda_1\lambda_2}_4$, we only consider 
	$F(\alpha_1,\dots,\alpha_5,\alpha,\beta) \in \G^*$ with $\alpha_4 = 0 \neq \alpha_2\alpha_5$ or $ \alpha_4 \neq 0 \neq \alpha_3^2+\alpha_5^2$, 
	the remaining parameters are arbitrary. There are three cases as follows.
	\begin{itemize}
		\item The first case is $\alpha_4 = 0 \neq \alpha_2\alpha_5$. Obviously, each of $x^*_1, x^*_2, x^*, y^*$ runs over line $\R$, 
		while $x^*_4 \equiv 0$, and $x^*_2, x^*_5 \in \R$ satisfying $\alpha_2 x^*_2 > 0$ and $\alpha_5 x^*_5 >0$.
		Hence, $\Omega_F$ is a part of hyperplane as follows:
		\[
			\Omega_F = \left\lbrace (x^*_1, x^*_2, x^*_3, x^*_4, x^*_5, x^*, y^*) \in \G^* \, \big| \, x^*_4 = 0,  \alpha_2 x^*_2 > 0, \alpha_5 x^*_5 >0 \right\rbrace.
		\]

		\item The second case is $\alpha_5=0\neq\alpha_3\alpha_4$. Obviously, each of $x^*_1, x^*_2, x^*, y^*$ runs over line $\R$, 
		while $x^*_5 \equiv 0$, and $x^*_3, x^*_4 \in \R$ satisfying $\alpha_3 x^*_3 > 0$ and $\alpha_4 x^*_4 >0$. 
		For this reason, $\Omega_F$ is a part of hyperplane as follows:
		\[
			\Omega_F= \left\lbrace (x^*_1, x^*_2, x^*_3, x^*_4, x^*_5, x^*, y^*) \in \G^* \, \big| \,  x^*_5 = 0, \alpha_3 x^*_3 > 0, \alpha_4 x^*_4 >0 \right\rbrace.
		\]

		\item The final case is $\alpha_4 \alpha_5 \neq 0$. Clearly, each of $x^*_1, x^*, y^*$ runs over line $\R$, 
		while $x^*_4, x^*_5 \in \R$ satisfying $\alpha_4 x^*_4 > 0$ and $\alpha_5 x^*_5 >0$. 
		The coordinates $x^*_2, x^*_3, x^*_4, x^*_5$ satisfy the following equation
		\[
			\left( x^*_2-\frac{x^*_3x^*_4}{x^*_5}\right) \frac{{x^*_4}^{\frac{1+\lambda_1}{\lambda_2-\lambda_1-1}}}{{x^*_5}^{\frac{1}{\lambda_2-\lambda_1-1}}}
			= \left( \alpha_2-\frac{\alpha_3\alpha_4}{\alpha_5}\right) 
				\frac{{\alpha_4}^{\frac{1+\lambda_1}{\lambda_2-\lambda_1-1}}}{{\alpha_5}^{\frac{1}{\lambda_2-\lambda_1-1}}}.
		\]
		Thus, $\Omega_F$ is a part of hypersurface of degree two as follows:
		\begin{flalign*}
			\Omega_F = \left\lbrace (x^*_1, x^*_2, x^*_3, x^*_4, x^*_5, x^*, y^*) \in \G^* \; \Bigg| 
				\left( x^*_2-\frac{x^*_3x^*_4}{x^*_5}\right) \frac{{x^*_4}^{\frac{1+\lambda_1}{\lambda_2-\lambda_1-1}}}{{x^*_5}^{\frac{1}{\lambda_2-\lambda_1-1}}} = \right. \\	
				\left. = \left( \alpha_2-\frac{\alpha_3\alpha_4}{\alpha_5}\right) 
					\frac{{\alpha_4}^{\frac{1+\lambda_1}{\lambda_2-\lambda_1-1}}}{{\alpha_5}^{\frac{1}{\lambda_2-\lambda_1-1}}},
					\; \alpha_4 x^*_4 > 0, \, \alpha_5 x^*_5 >0 \right\rbrace
		\end{flalign*}
	\end{itemize}
This completes the proof for $G = G^{\lambda_1\lambda_2}_4$.

Let $G = G^{\lambda}_{12}$. In Step 1, by direct computations, the matrix of the bilinear form $B_F$ is as follows
\[
B_F = \begin{bmatrix}
0 & \alpha_4 & \alpha_5 & 0 & 0 & -\alpha_1 & 0 
\\
-\alpha_4 & 0 & 0 & 0 & 0 & -\lambda\alpha_2 & -(\alpha_2+\alpha_3)
\\
-\alpha_5 & 0 & 0 & 0 & 0 & -\lambda\alpha_3 & -\alpha_3 
\\
0 & 0 & 0 & 0 & 0 & -(1+\lambda)\alpha_4 & -(\alpha_4+\alpha_5)
\\
0 & 0 & 0 & 0 &0 & -(1+\lambda)\alpha_5 & -\alpha_5
\\
\alpha_1 & \lambda\alpha_2 & \lambda\alpha_3 & (1+\lambda)\alpha_4 & (1+\lambda)\alpha_5 & 0 & 0
\\
0 & \alpha_2+\alpha_3 & \alpha_3 & \alpha_4+\alpha_5 &  \alpha_5 & 0 & 0
\end{bmatrix}.
\]
For Step 2, it is obvious that the maximal rank of $B_F$ is six. Moreover
\[
\begin{array}{l l l}
\Rank (B_F) = 6 & \Leftrightarrow & \alpha_3\alpha_4 \neq 0 =\alpha_5 \mbox{  or } \alpha_5\neq0.
\end{array}
\]
	In Step 3, we also find that
\[\exp(\ad_U)=\begin{bmatrix}
e^{x} & 0 & 0 & 0 & 0 & -x_1p &0 
\\
0 & e^{\lambda x+y} & 0 & 0 & 0 &-\lambda x_2r & -x_2r 
\\
0 & ye^{\lambda x+y} &e^{\lambda x+y} & 0 & 0 &  b_{12} &f_{12}
\\
-x_2e^{x}r &x_1e^{\lambda x+y}p&0& e^{\xi} & 0  & c_{12} & g_{12} 
\\
a_{12} & x_1ye^{\lambda x+y}p & x_1e^{\lambda x+y}p & ye^{\xi} & e^{\xi} & d_{12} & h_{12}
\\
0 & 0 & 0 & 0 & 0 & 1 & 0
\\
0 & 0 & 0 & 0 & 0 & 0 & 1
\end{bmatrix} ,\]

where
\[ \begin{array}{l}
\xi=(1+\lambda) x+y, \quad 
p=\dfrac{e^{x}-1}{x}, \quad 
r=\dfrac{e^{\lambda x +y}-1}{\lambda x +y}\\
a_{12}, b_{12}, f_{12}: \text{ expressions  containing parameters }  x_{2}, x_{3}, x,  y\\
c_{12}, d_{12}, g_{12}, h_{12}: \text{  expressions containing  parameters }  x_{1}, x_{2}, x_{3}, x_{4}, x_{5}, x,  y.
\end{array}
\]
By a direct computation, we get
\[\begin{array}{l}
x^*_1=\alpha_1 e^{x}-\alpha_4 x_2e^{x}r + \alpha_5a_{12}
\\
x^*_2=\alpha_2e^{\lambda x+y}+ \alpha_3ye^{\lambda x+y}+ \alpha_4x_1e^{\lambda x+y}p +\alpha_5x_1ye^{\lambda x+y}p 
\\
x^*_3=  \alpha_3e^{\lambda x+y} +\alpha_5 x_1e^{\lambda x+y}p 
\\
x^*_4= \alpha_4e^{\xi} + \alpha_5ye^{\xi}
\\
x^*_5= \alpha_5 e^{\xi}
\\
x^*=  -\alpha_1x_1p -\alpha_2\lambda x_2r +\alpha_3 b_{12}  +\alpha_4c_{12} + \alpha_5  d_{12}+ \alpha 
\\
y^*=-\alpha_2x_2r + \alpha_3f_{12} +\alpha_4g_{12} +\alpha_5 h_{12}+\beta.
\end{array}\]
To describe all maximal dimensional K-orbits of $G = G^{\lambda}_{12}$, we only consider $F(\alpha_1,\dots,\alpha_5,\alpha,\beta) \in \mathcal{G}^{*}$ with $\alpha_3\alpha_4 \neq 0 =\alpha_5$ or $\alpha_5\neq0$, the remaining parameters are arbitrary.

\begin{itemize}
	\item 	The second case, $\alpha_3\alpha_4 \neq 0=\alpha_5$. Obviously, each of $x^*_1, x^*_2, x^*, y^*$ runs over line ${\mathbb R}$, while $x^*_5 \equiv 0$ and $x^*_3, x^*_4 \in {\mathbb R}; \, \alpha_3 x^*_3 > 0, \, \alpha_4 x^*_4 >0$. For this reason, we get	 	
	$\Omega_F$ is a part of hyperplane as follows:
	\[\begin{array}{r} \Omega_F=\big\{(x^*_1, x^*_2, x^*_3, x^*_4, x^*_5, x^*, y^*) \in \mathcal{G}^*:   x^*_5 = 0, \, \alpha_3 x^*_3 > 0, \, \alpha_4 x^*_4 >0 \big\}.\end{array}\]
	
	\item The final case, $\alpha_5 \neq 0$. Clearly, each of $x^*_1, x^*, y^*$ runs over line ${\mathbb R}$, while $ x^*_5 \in {\mathbb R}; \,  \alpha_5 x^*_5 >0$. By an easy computation it follows that the coordinates $x^*_2, x^*_3, x^*_4, x^*_5$ are satisfy the following equation
	\[ 	\left(x^*_2-\frac{x^*_3x^*_4}{x^*_5}\right)\frac{1}{{x^*_5}^{\frac{\lambda}{1+\lambda}}e^{\frac{x^*_4}{(1+\lambda)x^*_5}}}  =\left(\alpha_2-\frac{\alpha_3\alpha_4}{\alpha_5}\right)\frac{1}{\alpha_5^{\frac{\lambda}{1+\lambda}}e^{\frac{\alpha_4}{(1+\lambda)\alpha_5}}}\cdot\]
	Thus, we get	 	
	$\Omega_F$ is a half of hypersurface of degree two as follows:
	\[\begin{array}{rl}
	\Omega_F=\Bigg\{(x^*_1, x^*_2, x^*_3, x^*_4, x^*_5, x^*, y^*) \in \mathcal{G}^*: &	\left(x^*_2-\frac{x^*_3x^*_4}{x^*_5}\right)\frac{1}{{x^*_5}^{\frac{\lambda}{1+\lambda}}e^{\frac{x^*_4}{(1+\lambda)x^*_5}}} =\\& =\left(\alpha_2-\frac{\alpha_3\alpha_4}{\alpha_5}\right)\frac{1}{\alpha_5^{\frac{\lambda}{1+\lambda}}e^{\frac{\alpha_4}{(1+\lambda)\alpha_5}}},
	\, \alpha_5 x^*_5 >0 \Bigg\}.\end{array}\]
\end{itemize}
This completes the proof for $G = G^{\lambda}_{12}$.

Let $G=G^{\lambda}_{13}$. In Step 1, by direct computations, the matrix of the bilinear form $B_F$ is as follows
\[
B_F = \begin{bmatrix}
0 & \alpha_4 & \alpha_5 & 0 & 0 & 0 & -\lambda\alpha_1 
\\
-\alpha_4 & 0 & 0 & 0 & 0 & -\alpha_2 & -\alpha_3
\\
-\alpha_5 & 0 & 0 & 0 & 0 & -\alpha_3 & \alpha_2
\\
0 & 0 & 0 & 0 & 0 & -\alpha_4 & -(\lambda\alpha_4+\alpha_5)
\\
0 & 0 & 0 & 0 &0 & -\alpha_5 & \alpha_4-\lambda\alpha_5
\\
0 & \alpha_2 & \alpha_3 & \alpha_4 &  \alpha_5 & 0 & 0
\\
\lambda\alpha_1 & \alpha_3 & -\alpha_2 & \lambda\alpha_4+\alpha_5 & -(\alpha_4-\lambda\alpha_5) & 0 & 0
\end{bmatrix}.
\]
For Step 2, it is obvious that the maximal rank of $B_F$ is six. Moreover
\[
\begin{array}{l l l}
\Rank (B_F) = 6 & \Leftrightarrow & \alpha^{2}_{4}+\alpha^{2}_{5} \neq 0.
\end{array}
\]
In Step 3, we also find that
\[\exp(\ad_U) =\begin{bmatrix}
e^{\lambda y} & 0 & 0 & 0 & 0 & 0 & h_{13} 
\\
0 & e^{x}\cos(y) & -e^{x}\sin(y) & 0 & 0 & c_{13} & m_{13} 
\\
0 & e^{x}\sin(y) & e^{x}\cos(y) & 0 & 0 & d_{13} & n_{13} 
\\
a_{13} & x_1e^{x}\cos(y) & -x_1e^{x}\sin(y) & e^{\xi}\cos(y) & -e^{\xi}\sin(y) & f_{13} & k_{13}
\\
b_{13} & x_1e^{x}\sin(y) & x_1e^{x}\cos(y)  & e^{\xi}\sin(y) & e^{\xi}\cos(y) & g_{13} & l_{13}
\\
0 & 0 & 0 & 0 & 0 & 1 & 0
\\
0 & 0 & 0 & 0 & 0 & 0 & 1
\end{bmatrix}  ,\]
where
\[\begin{array}{l}
\xi=x+\lambda y,\\
a_{13}, b_{13}, c_{13}, d_{13},  m_{13}, n_{13} \colon \text{expressions containing parameters }  x_{2}, x_{3}, x,  y,\\
f_{13}, g_{13}, k_{13}, l_{13} \colon \text{expressions containing parameters }  x_{1}, x_{2}, x_{3}, x_{4}, x_{5}, x,  y,\\
h_{13} \colon \text{an expression containing parameters } x_{1}, x,  y.
\end{array}
\]

By a direct computation, we get
\[\begin{array}{l}
x^*_1=\alpha_1e^{\lambda y} +\alpha_4 a_{13} + \alpha_5b_{13}
\\
x^*_2=\alpha_2e^{x}\cos(y) + \alpha_{3}e^{x}\sin(y)+ \alpha_4x_1e^{x}\cos(y)   +\alpha_{5}x_1e^{x}\sin(y)
\\
x^*_3=  -\alpha_2e^{x}\sin(y) + \alpha_{3}e^{x}\cos(y)- \alpha_4x_1e^{x}\sin(y)   +\alpha_{5}x_1e^{x}\cos(y)
\\
x^*_4= \alpha_4e^{x+\lambda y} \cos(y) + \alpha_45e^{x+\lambda y} \sin(y) 
\\
x^*_5= -\alpha_4e^{x+\lambda y} \sin(y) + \alpha_45e^{x+\lambda y} \cos(y) 
\\
x^*=  \alpha_2c_{13} +\alpha_3d_{13}  +\alpha_4f_{13} + \alpha_5 g_{13}+ \alpha 
\\
y^*=\alpha_{1}h_{13} + \alpha_2m_{13} + \alpha_3n_{13} +\alpha_4k_{13} + \alpha_5 l_{13} + \beta.
\end{array}\]

To describe all maximal-dimensional K-orbits of $G = G^{\lambda}_{13}$, we only consider $F(\alpha_1,\dots,\alpha_5,\alpha,\beta) \in \G^*$ 
with $\alpha_4^2+\alpha^2_5 \neq 0$, the remaining parameters are arbitrary.	
Clearly, $x^*_1, x^*, y^*$ run over line $\R$, while $x^*_4, x^*_5 \in \R; \,  {x^{*}_4}^2+{x^{*}_5}^2 \neq 0$.
The coordinates $x^*_2, x^*_3, x^*_4, x^*_5$ satisfy the following equation
\[ 	\frac{x^*_2x^*_5-x^*_3x^*_4}{{x^{*}_4}^2+{x^{*}_5}^2}e^{\lambda \arctan\frac{x^*_4}{x^*_5}}=\frac{\alpha_2\alpha_5-\alpha_3\alpha_4}{{\alpha_4}^2+{\alpha_5}^2}e^{\lambda\arctan\frac{\alpha_4}{\alpha_5}}\cdot\]
Thus, $\Omega_F$ is a part of hypersurface of degree two as follows:
\[\begin{array}{l}
\Omega_F=\Bigg\{(x^*_1, x^*_2, x^*_3, x^*_4, x^*_5, x^*, y^*) \in \mathcal{G}^*: 	\frac{x^*_2x^*_5-x^*_3x^*_4}{{x^{*}_4}^2+{x^{*}_5}^2}e^{\lambda \arctan\frac{x^*_4}{x^*_5}}=\frac{\alpha_2\alpha_5-\alpha_3\alpha_4}{{\alpha_4}^2+{\alpha_5}^2}e^{\lambda\arctan\frac{\alpha_4}{\alpha_5}},\,
{x^{*}_4}^2+{x^{*}_5}^2\neq 0 \Bigg\}.\end{array}\]
This completes the proof for $G = G^{\lambda}_{13}$.

The remaining cases can be proved similarly while the matrices of $B_{F}$  with necessary and sufficient conditions for $\Rank B_F=6$ and $\exp(\ad_U)$ are given in Tables \ref{tab5.2} and \ref{tab5.3}, respectively.

 \begin{longtable}{p{.060\textwidth} p{.550\textwidth} p{.260\textwidth}}
	\caption{ $B_{F}$ and Condition of $\Rank(B_{F})=6$ }\label{tab5.2} \\
	\hline No. & $B_{F}$, \quad $F(\alpha_1,\dots,\alpha_5,\alpha,\beta) \in \mathcal{G}^*$ & Condition of $\Rank(B_{F})=6$ \endfirsthead 
	\caption{ (continued)} \\ 
	\hline No. & $B_{F}$, \quad $F(\alpha_1,\dots,\alpha_5,\alpha,\beta) \in \mathcal{G}^*$ &  Condition of $\Rank(B_{F})=6$ \\ \hline \endhead \hline
	\endlastfoot \hline
	$\mathcal{G}^{\lambda}_{1}$ & $\begin{bmatrix}
	0 & \alpha_4 & \alpha_5 & 0 & 0 & - \alpha_1 & 0 
	\\
	- \alpha_4 & 0 & 0 & 0 & 0 & \alpha_2 & 0 
	\\
	- \alpha_5 & 0 & 0 & 0 & 0 & 0 & - \alpha_3 
	\\
	0 & 0 & 0 & 0 & 0 & 0 & 0
	\\
	0 & 0 & 0 & 0 & 0 & - \alpha_5 & - \alpha_5
	\\
	\alpha_1 & -\alpha_2 & 0 & 0 & \alpha_5 & 0 & \alpha_4
	\\
	0 & 0 & \alpha_3 & 0 & \alpha_5 & -\alpha_4 & 0
	\end{bmatrix}$ &  $\alpha_5 \neq 0 \neq \alpha_2^2+\alpha_4^2$\\
	$\mathcal{G}_{5}$ & $\begin{bmatrix}
	0 & \alpha_4 & \alpha_5 & 0 & 0 & 0 & -(\alpha_1+\alpha_2)
	\\
	-\alpha_4 & 0 & 0 & 0 & 0 & 0 & -\alpha_2 
	\\
	-\alpha_5 & 0 & 0 & 0 & 0 & -\alpha_3 & 0
	\\
	0 & 0 & 0 & 0 &0 & 0 & -2\alpha_4
	\\
	0 & 0 & 0 & 0 & 0 & -\alpha_5 & -\alpha_5
	\\
	0 & 0 & \alpha_3 & 0 & \alpha_5 & 0 & 0
	\\
	\alpha_1+\alpha_2& \alpha_2 & 0 & 2\alpha_4 & \alpha_5 & 0 & 0
	\end{bmatrix} $ & $\begin{array}{l}
	\alpha_4 = 0 \neq \alpha_2\alpha_5 \\\mbox{  or } \alpha_4 \neq 0 \neq  \alpha_3^2+\alpha_5^2
	\end{array}$
	\\
	$\mathcal{G}^{\lambda}_{6}$ & $\begin{bmatrix}
	0 & \alpha_4 & \alpha_5 & 0 & 0 & -\alpha_1 & -\alpha_2 
	\\
	-\alpha_4 & 0 & 0 & 0 & 0 & -\alpha_2 & 0 
	\\
	-\alpha_5 & 0 & 0 & 0 & 0 & -\lambda\alpha_3 & -\alpha_3
	\\
	0 & 0 & 0 & 0 &0 & -2\alpha_4 & 0
	\\
	0 & 0 & 0 & 0 & 0 & -(1+\lambda)\alpha_{5} & -\alpha_5
	\\
	\alpha_1 & \alpha_2 & \lambda\alpha_3 & 2\alpha_4 & (1+\lambda)\alpha_{5} & 0 & 0
	\\
	\alpha_2 & 0 & \alpha_3 & 0 & \alpha_5 & 0 & 0
	\end{bmatrix} $ & $\begin{array}{l}
	\alpha_4 = 0 \neq \alpha_2\alpha_5 \\\mbox{  or } \alpha_4 \neq 0 \neq  \alpha_3^2+\alpha_5^2
	\end{array}$
	\\
	$\mathcal{G}_{7}$ & $\begin{bmatrix}
	0 & \alpha_4 & \alpha_5 & 0 & 0 & 0 & -\alpha_1
	\\
	-\alpha_4 & 0 & 0 & 0 & 0 & -\alpha_2 & -(\alpha_2+\alpha_5) 
	\\
	-\alpha_5 & 0 & 0 & 0 & 0 & -\alpha_3 & 0
	\\
	0 & 0 & 0 & 0 &0 & -\alpha_4 & -2\alpha_4
	\\
	0 & 0 & 0 & 0 & 0 & -\alpha_5 & -\alpha_5
	\\
	0 & \alpha_2 & \alpha_3 & \alpha_4 & \alpha_5 & 0 & 0
	\\
	\alpha_1 & \alpha_2+\alpha_5 & 0 & 2\alpha_4 & \alpha_5 & 0 & 0
	\end{bmatrix} $ & $\begin{array}{l}
	\alpha_4^2+\alpha^2_5\neq0 \\\mbox{  except to  }\\ \alpha_3=\alpha_5=0\neq \alpha_4
	\end{array}$
	\\
	$\mathcal{G}^{\lambda}_{8}$ & $\begin{bmatrix}
	0 & \alpha_4 & \alpha_5 & 0 & 0 & -\alpha_1 & 0
	\\
	-\alpha_4 & 0 & 0 & 0 & 0 & -\gamma & -(\alpha_2+\alpha_5) 
	\\
	-\alpha_5 & 0 & 0 & 0 & 0 & -\lambda\alpha_3 & \alpha_3
	\\
	0 & 0 & 0 & 0 &0 & -\delta & -\alpha_4
	\\
	0 & 0 & 0 & 0 & 0 & -\xi & -\alpha_5
	\\
	\alpha_1 & \gamma & \lambda\alpha_3 & \delta & \xi & 0 & 0
	\\
	0 & \alpha_2+\alpha_5 & \alpha_3 & \alpha_4 & \alpha_5 & 0 & 0
	\end{bmatrix} $ & $\begin{array}{l}
	\alpha_4^2+\alpha^2_5\neq0 \\\mbox{  except to  }\\ \alpha_3=\alpha_5=0\neq \alpha_4
	\end{array}$
	\\ \hline
	$\mathcal{G}_{11}$ & $\begin{bmatrix}
	0 & \alpha_4 & \alpha_5 & 0 & 0 & 0 & -\alpha_1 
	\\
	-\alpha_4 & 0 & 0 & 0 & 0 & -\alpha_2 & -\alpha_3
	\\
	-\alpha_5 & 0 & 0 & 0 & 0 & -\alpha_3 & 0 
	\\
	0 & 0 & 0 & 0 & 0 & -\alpha_4 & -(\alpha_4+\alpha_5 )
	\\
	0 & 0 & 0 & 0 &0 & -\alpha_5 & -\alpha_5
	\\
	0 & \alpha_2 & \alpha_3 & \alpha_4 & \alpha_5 & 0 & 0
	\\
	\alpha_1 & \alpha_3 & 0 & \alpha_4+\alpha_5  &  \alpha_5 & 0 & 0
	\end{bmatrix} $ & $ \alpha_4 \neq 0 \neq {\alpha_3}^2 + {\alpha_5}^2 $ \\
	$\mathcal{G}^{\lambda_1,\lambda_2}_{14}$ & $\begin{bmatrix}
	0 & \alpha_4 & \alpha_5 & 0 & 0 & -\alpha_1  & 0
	\\
	-\alpha_4 & 0 & 0 & 0 & 0 & -\lambda_1\alpha_2 & -\beta
	\\
	-\alpha_5 & 0 & 0 & 0 & 0 & -\lambda_{1}\alpha_3 & -\gamma
	\\
	0 & 0 & 0 & 0 & 0 & -\delta & -\xi
	\\
	0 & 0 & 0 & 0 &0 & -\sigma& -\chi
	\\
	\alpha_{1} & \lambda_{1}\alpha_2 & \lambda_{1}\alpha_3 & \delta&  \sigma & 0 & 0
	\\
	0 & \beta & \gamma& \xi  & \chi & 0 & 0
	\end{bmatrix} $ & $
	\alpha^{2}_{4}+\alpha^{2}_{5} \neq 0 $ \\
	& where $\begin{array}{ll}
	\beta=\lambda_{2}\alpha_2+\alpha_3,& \gamma=-\alpha_2+\lambda_{2}\alpha_{3}  \\
	\delta=(1+\lambda_{1})\alpha_4, & \sigma=(1+\lambda_{1})\alpha_5
	\\
	\xi= \lambda_{2}\alpha_{4}+\alpha_{5}, & \chi= -\alpha_{4}+\lambda_{2}\alpha_{5}
	\end{array}$&\\
	$\mathcal{G}_{15}$ & $\begin{bmatrix}
	0 & \alpha_4 & \alpha_5 & 0 & 0 & 0 & 0 
	\\
	-\alpha_4 & 0 & 0 & 0 & 0 & -\alpha_3 & -\beta
	\\
	-\alpha_5 & 0 & 0 & 0 & 0 & \alpha_2 & -\gamma
	\\
	0 & 0 & 0 & 0 & 0 & -\alpha_5 & -\alpha_4
	\\
	0 & 0 & 0 & 0 &0 & \alpha_4 & -\alpha_5
	\\
	0 & \alpha_3 & -\alpha_2 & \alpha_5 &  -\alpha_4 & 0 & 0
	\\
	0 & \beta & \gamma & \alpha_4 & \alpha_5 & 0 & 0
	\end{bmatrix} $ & $
	\alpha^{2}_{4}+\alpha^{2}_{5} \neq 0 $ \\
	& where $ \beta=\alpha_2+\alpha_5, \gamma=\alpha_3-\alpha_4 $&\\	
	$\mathcal{G}^{\lambda}_{16}$ & $\begin{bmatrix}
	0 & \alpha_4 & \alpha_5 & 0 & 0 & 0 & 0 
	\\
	-\alpha_4 & 0 & 0 & 0 & 0 & -\beta & -\gamma
	\\
	-\alpha_5 & 0 & 0 & 0 & 0 & \alpha_2 & -\sigma
	\\
	0 & 0 & 0 & 0 & 0 & -\alpha_5 & -\alpha_4
	\\
	0 & 0 & 0 & 0 &0 & \alpha_4 & -\alpha_5
	\\
	0 & \beta & -\alpha_2 & \alpha_5 &  -\alpha_4 & 0 & 0
	\\
	0 & \gamma & \sigma & \alpha_4 & \alpha_5 & 0 & 0
	\end{bmatrix} $ & $
	\alpha^{2}_{4}+\alpha^{2}_{5} \neq 0 $ \\
	& where $\begin{array}{ll}
	\beta=\alpha_3+\alpha_5,& \gamma=\alpha_2+\lambda\alpha_5,\\ \sigma=\alpha_3-\lambda\alpha_4 &
	\end{array}$&\\
	\hline
\end{longtable}

\begin{longtable}{p{.060\textwidth}  p{.800\textwidth}}
	\caption{  $\exp(\ad_U)$ of considered Lie groups}\label{tab5.3} \\
	\hline No. & $\exp(\ad_U), \; U=\sum_{i=1}^{5}x_iX_i+xX+yY  \; (x_i, x, y \in \R)$  \endfirsthead 
	\caption{(continued)} \\ 
	\hline No. & $\exp(\ad_U), \; U=\sum_{i=1}^{5}x_iX_i+xX+yY \; (x_i, x, y \in \R)$   \\ \hline \endhead \hline
	\endlastfoot \hline
	$ G_1^{\lambda} $	& $ 	\begin{bmatrix}
	e^{x} & 0     & 0 & 0 & 0 & -x_1p & 0
	\\
	0      & e^{-x} & 0 & 0 & 0 &x_2e^{-x} p & 0 
	\\
	0      & 0     & e^{y} & 0 & 0 & 0 & -x_3q
	\\
	-x_2p & x_1e^{-x}p& 0 & 1 &0 & a_{1} & \lambda x
	\\
	-x_3e^xq & 0 & x_1e^{y}p & 0 & e^{x+y} & b_{1} &c_{1}
	\\
	0 & 0 & 0 & 0 & 0 & 1 & 0
	\\
	0 & 0 & 0 & 0 & 0 & 0 & 1
	\end{bmatrix} $\\
	& where $  \begin{cases}
	p=\dfrac{e^{x}-1}{x}, \quad 
	q=\dfrac{e^{y}-1}{y}\\
	a_{1}: \text{ an expression containing parameters } x_{1}, x_{2}, x, y\\
	b_{1}, c_{1}:  \text{ expressions  containing parameters } x_{1}, x_{3}, x_{5},  x, y.
	\end{cases}$ \\
	$G_5 $ & $\begin{bmatrix}
	e^{y} & 0 & 0 & 0 & 0 & 0 & -x_1q
	\\
	e^{y}y & e^{y} & 0 & 0 & 0 &0 & c_{5}
	\\
	0 & 0 & e^{x} & 0 & 0 & -x_3p & 0 
	\\
	a_{5}&x_1e^yq & 0& e^{2y} &0 &0 & d_{5}
	\\
	-x_3e^{y}p & 0 & x_1e^{x}q & 0 & e^{x+y} & b_{5} & f_{5}
	\\
	0 & 0 & 0 & 0 & 0 & 1 & 0
	\\
	0 & 0 & 0 & 0 & 0 & 0 & 1
	\end{bmatrix} $ \\
	& where $ \begin{cases}
	a_{5}, c_{5}: \text{ expressions  containing parameters } x_{1}, x_{2}, x,  y\\
	d_{5}: \text{ an expression containing parameters } x_{1},  x_{4},  y\\
	b_{5}, f_{5}: \text{ expressions  containing parameters } x_{1},  x_{3}, x_{5}, x,  y.
	\end{cases} $\\
	$ G_6^{\lambda} $ & $\begin{bmatrix}
	e^{x} & 0 & 0 & 0 & 0  & -x_1p & 0 
	\\
	e^{x}y & e^{x} & 0 & 0 & 0 &  b_{6}    & -x_1p 
	\\
	0 & 0 & e^{\lambda x+y} & 0 & 0 & -\lambda x_3r &  -x_3r
	\\
	a_{6}     &x_1e^xp & 0& e^{2x} &0 &  c_{6}   &  f_{6}   
	\\
	-x_3e^xr & 0 & x_1e^{\lambda x+y}p & 0 & e^{\xi} &  d_{6}  &     g_{6}
	\\
	0 & 0 & 0 & 0 & 0 & 1 & 0
	\\
	0 & 0 & 0 & 0 & 0 & 0 & 1
	\end{bmatrix} $ \\
	& where $ \begin{cases}
	\xi = \lambda x+x+y, \quad 
	r=\dfrac{e^{\lambda x +y}-1}{\lambda x +y}\\
	a_{6}, b_{6}: \text{ expressions  containing parameters } x_{1}, x_{2}, x,  y\\
	c_{6}: \text{ an expression  containing parameters } x_{1},  x_{4}, x,  y\\
	f_{6}: \text{ an expression  containing parameters } x_{1}, x\\
	d_{6}, g_{6}: \text{ expressions containing parameters } x_{1},  x_{3}, x_{5}, x,  y.
	\end{cases} $\\
	$G_7 $ & $\begin{bmatrix}
	e^{y} & 0 & 0 & 0 & 0 & 0 & -x_1q
	\\
	0 & e^{x+y} & 0 & 0 & 0 & b_{7} & b_{7} 
	\\
	0 & 0 & e^{x} & 0 & 0 & -x_3p  & 0 
	\\
	a_{7}&x_1e^{x+y}q & 0& e^{x+2y} &0 &c_{7} &f_{7}
	\\
	-x_3e^{y}p & ye^{x+y}& x_1e^{x} & 0 & e^{x+y} &d_{7} & g_{7}
	\\
	0 & 0 & 0 & 0 & 0 & 1 & 0
	\\
	0 & 0 & 0 & 0 & 0 & 0 & 1
	\end{bmatrix} $\\
	& where $ \begin{cases}
	a_{7}, b_{7}: \text{ expressions containing parameters }  x_{2}, x,  y\\
	c_{7}, f_{7}: \text{ expressions containing parameters } x_{1}, x_{2}, x_{4}, x,  y\\
	d_{7}, g_{7}: \text{ expressions  containing parameters } x_{1}, x_{2}, x_{3}, x_{5}, x,  y.
	\end{cases} $\\
	$G_8^{\lambda}  $ & $ \begin{bmatrix}
	e^{x} & 0 & 0 & 0 & 0 & -x_1p & 0 
	\\
	0 & e^{\xi} & 0 & 0 & 0 & -(1+\lambda)x_2s & -x_2s \\
	0 & 0 & e^{ \lambda x+y} & 0 & 0 & -\lambda x_3r  & -x_3r
	\\
	-x_2e^{x}s&x_1e^{\xi}p & 0& e^{x+\xi} &0 &a_{8} &c_{8}
	\\
	-x_3e^{x}r & ye^{\xi}& x_1e^{ \lambda x+y}p & 0 & e^{\xi} &b_{8} & d_{8}
	\\
	0 & 0 & 0 & 0 & 0 & 1 & 0
	\\
	0 & 0 & 0 & 0 & 0 & 0 & 1
	\end{bmatrix} $\\
	& where $ \begin{cases}
	s=\dfrac{e^{\lambda x +x +y}-1}{\lambda x + x +y}\\
	a_{8}, c_{8}: \text{ expressions  containing parameters } x_{1}, x_{2}, x_{4}, x,  y\\
	b_{8}, d_{8}: \text{ expressions  containing parameters } x_{1}, x_{2}, x_{3}, x_{5}, x,  y.
	\end{cases} $\\
	$G_{11}  $ & $\begin{bmatrix}
	e^{y} & 0 & 0 & 0 & 0 & 0 &-x_1q\\
	0 & e^{x} & 0 & 0 & 0 &-x_2p & 0 \\
	0 & ye^{x} &e^{x} & 0 & 0 &  b_{11} &-x_2p\\
	-x_2e^{y}p &x_1e^{x}q & 0 & e^{x+y} & 0  & c_{11} & f_{11} \\
	a_{11} & x_1e^{x}(e^{y}-1)& x_1e^{x}q &ye^{x+y} & e^{x+y}  & d_{11} & g_{11}\\
	0 & 0 & 0 & 0 & 0 & 1 & 0\\
	0 & 0 & 0 & 0 & 0 & 0 & 1
	\end{bmatrix}  $\\
	& where $ \begin{cases}
	a_{11}, b_{11}: \text{ expressions  containing parameters }  x_{2}, x_{3}, x,  y\\
	c_{11}, f_{11}: \text{ expressions  containing parameters }  x_{1}, x_{2},  x_{4}, x,  y\\
	d_{11}, g_{11}: \text{ expressions  containing parameters }  x_{1}, x_{2}, x_{3}, x_{4}, x_{5}, x,  y.
	\end{cases} $\\
	$G_{14}^{\lambda_1,\lambda_2}  $ & $\begin{bmatrix}
	e^{x} & 0 & 0 & 0 & 0 & c_{14} & 0 
	\\
	0 & e^{\zeta}\cos(y) & -e^{\zeta}\sin(y) & 0 & 0 & d_{14} & m_{14}
	\\
	0 & e^{\zeta}\sin(y) & e^{\zeta}\cos(y) & 0 & 0 & f_{14} & n_{14} 
	\\
	a_{14} & x_1e^{\zeta}\cos(y) & -x_1e^{\zeta}\sin(y) & e^{\zeta+x}\cos(y) & -e^{\zeta+x}\sin(y) & g_{14} & k_{14}
	\\
	b_{14} & x_1e^{\zeta}\sin(y)  & x_1e^{\zeta}\cos(y)  & e^{\zeta+x}\sin(y) & e^{\zeta+x}\cos(y) & h_{14} & l_{14}
	\\
	0 & 0 & 0 & 0 & 0 & 1 & 0
	\\
	0 & 0 & 0 & 0 & 0 & 0 & 1
	\end{bmatrix}  $\\
	&where $ \begin{cases}
	\zeta=\lambda_1x+\lambda_2y\\
	c_{14}: \text{ an expression containing parameters } x_{1}, x,  y\\
	a_{14}, b_{14}, d_{14}, f_{14},  m_{14}, n_{14}: \text{ expressions  containing parameters }  x_{2}, x_{3}, x,  y\\
	g_{14}, h_{14}, k_{14}, l_{14}: \text{ expressions  containing parameters }  x_{1}, x_{2}, x_{3}, x_{4}, x_{5}, x,  y.
	\end{cases} $\\
	$G_{15}  $ & $\begin{bmatrix}
	1 & 0 & 0 & 0 & 0 & 0 & 0 
	\\
	0 & e^{y}\cos(x) & -e^{y}\sin(x) & 0 & 0 & c_{15} & m_{15} 
	\\
	0 & e^{y}\sin(x) & e^{y}\cos(x) & 0 & 0 & d_{15} & n_{15} 
	\\
	a_{15} & p_{15} & q_{15} & e^{y}\cos(x) &-e^{y}\sin(x) & f_{15} & k_{15}
	\\
	b_{15} & r_{15} & p_{15} & e^{y}\sin(x) & e^{y}\cos(x) & g_{15} & l_{15}
	\\
	0 & 0 & 0 & 0 & 0 & 1 & 0
	\\
	0 & 0 & 0 & 0 & 0 & 0 & 1
	\end{bmatrix}  $\\
	&where $ \begin{cases}
	p_{15}= (x_1\cos(x)-y\sin(x))e^{y} 
	\\
	q_{15}=(-x_1\sin(x)-y\cos(x))e^{y}
	\\
	r_{15}= (x_1\sin(x)+y\cos(x))e^{y}   
	\\
	a_{15}, b_{15}, c_{15}, d_{15},  m_{15}, n_{15}: \text{ expressions  containing parameters }  x_{2}, x_{3}, x,  y\\
	f_{15}, g_{15}, , k_{15}, l_{15}: \text{ expressions  containing parameters } x_{1}, x_{2}, x_{3}, x_{4}, x_{5}, x,  y.
	\end{cases} $\\
	$ G^{\lambda}_{16} $ & $\begin{bmatrix}
	1 & 0 & 0 & 0 & 0 & 0 & 0 
	\\
	0 & e^{y}\cos(x) & -e^{y}\sin(x) & 0 & 0 & c_{16} & m_{16} 
	\\
	0 & e^{y}\sin(x) & e^{y}\cos(x) & 0 & 0 & d_{16} & n_{16} 
	\\
	a_{16} & p_{16} & q_{16} & e^{y}\cos(x) &-e^{y}\sin(x) & f_{16} & k_{16}
	\\
	b_{16} & r_{16} & p_{16} & e^{y}\sin(x) & e^{y}\cos(x) & g_{16} & l_{16}
	\\
	0 & 0 & 0 & 0 & 0 & 1 & 0
	\\
	0 & 0 & 0 & 0 & 0 & 0 & 1
	\end{bmatrix}  $\\
	&where $ \begin{cases}
	p_{16}= \left( x_1\cos(x)-\dfrac{1}{2}x\sin(x)\right) e^{y} -\lambda y\sin(x)e^y
	\\
	q_{16}=\left( -x_1\sin(x)-\dfrac{1}{2}x\cos(x)+\dfrac{1}{2}\sin(x)\right) e^{y}-\lambda y\cos(x)e^y
	\\
	r_{16}= \left( x_1\sin(x)+\dfrac{1}{2}x\cos(x)+\dfrac{1}{2}\sin(x)\right) e^{y}+\lambda y\cos(x)e^y  
	\\
	a_{16}, b_{16}, c_{16}, d_{16},  m_{16}, n_{16}: \text{ expressions  containing parameters }  x_{2}, x_{3},  x,  y\\
	f_{16}, g_{16}, k_{16}, l_{16}: \text{ expressions  containing parameters } x_{1}, x_{2}, x_{3}, x_{4}, x_{5}, x,  y.
	\end{cases} $\\
	\hline
\end{longtable}
\end{proof}

\begin{rem}[Geometry of maximal-dimensional K-orbits]\label{RemarkOrbits}
	Due to the picture of K-orbits of maximal dimension of considered Lie groups, we have some geometric characteristics as~follows.
	\begin{enumerate}
		\item All K-orbits of $G \in \A$ are strictly homogeneous simplectic submanifolds of $\G^*$ (see \cite[\S15.1]{kir}). 
		Moreover, if $G$ is exponential, i.e.  $G \in \A_1$, 
		all K-orbits of $G$ are homeomorphic to Euclidean spaces (see \cite[\S15.3]{kir}). 
		It can be easily verified by using the picture of K-orbits in Theorem \ref{picture K-orbit}.

		\item For each $G \in \A_1$,  
		there are exactly two types of maximal-dimensional K-orbits. 
		\begin{enumerate}[2.1]
			\item For $G = G_1^{\lambda}$, Type 1 has exactly four K-orbits but Type 2 has an infinite family of K-orbits. Namely,
			\begin{itemize}
				\item Each orbit of Type 1 is a quarter of hyperplane $\{x_4^* = 0\}$ in $\G^*$ which is obtained when ``cutting" 
				it by two hyperplanes $\{x_2^* = 0\}$ and $\{x_5^* = 0\}$. In fact, K-orbits of Type 1 are four connected components of 
				$\{x_4^* = 0\} \setminus \left(\{x_2^* = 0\} \cup \{x_5^* = 0\}\right)$.
				
				\item Each orbit of Type 2 is a part of hyperplane $\{x_4^* = c\}$ ($c$ is a non-zero constant) which is obtained 
				when ``cutting" it by two hyperplanes $\{x_4^* = 0\}$ and $\{x_5^* = 0\}$.
			\end{itemize}
		
		\item For $G=G_{12}^\lambda$, Type 1 has exactly four K-orbits but Type 2 has an infinite family of K-orbits. Namely,
		\begin{itemize}
			\item Each orbit of Type 1 is a quarter of hyperplane $\{x_5^* = 0\}$ in $\G^*$ which is obtained when ``cutting" 
			it by two hyperplanes $\{x_3^* = 0\}$ and $\{x_4^* = 0\}$. In fact, K-orbits of Type 1 are four connected components 
			of $\{x_5^* = 0\} \setminus \left(\{x_3^* = 0\} \cup \{x_4^* = 0\}\right)$.  
			
			\item Each orbit of Type 2 is a half of an algebraic hypersurface in $\G^*$ which is obtained when ``cutting" it by hyperplane $\{x_5^* = 0\}$.
		\end{itemize}
					 
			\item For each $G \in \A_1 \setminus \{G_1^{\lambda}, G_{12}^\lambda \}$, Type 1 has a finite family of K-orbits but Type 2 has an infinite family of K-orbits. Namely,
			\begin{itemize}
				\item Type 1:
				\begin{itemize}
					\item [(i)] If $G \in \{G_2, G_{11}\}$, each orbit of Type 1 is a quarter of hyperplane $\{x_5^* = 0\}$ in $\G^*$ which is obtained when ``cutting" 
					it by two hyperplanes $\{x_3^* = 0\}$ and $\{x_4^* = 0\}$. 
					\item [(ii)] If $G \in \left\lbrace G_7, G^{\lambda}_{8}\right\rbrace $, each orbit of Type 1 is either a haft of hyperplane $\{x_4^* = 0\}$ in $\G^*$ which is obtained when ``cutting" 
					it by  hyperplane $\{x_5^* = 0\}$ or a quarter of hyperplane $\{x_5^* = 0\}$ in $\G^*$ which is obtained when ``cutting" 
					it by two hyperplanes $\{x_3^* = 0\}$ and $\{x_4^* = 0\}$. 
					\item [(iii)] If $G \in \left\lbrace  G_3, G_4^{\lambda_1\lambda_2}, G_5, G_6^{\lambda}, G_9, G_{10}^{\lambda}\right\rbrace $, each orbit of Type 1 is either a quarter  of hyperplane $\{x_4^* = 0\}$ in $\G^*$ which is obtained when ``cutting" 
					it by two hyperplanes $\{x_2^* = 0\}$ and $\{x_5^* = 0\}$ or a quarter of hyperplane $\{x_5^* = 0\}$ in $\G^*$ which is obtained when ``cutting" 
					it by two hyperplanes $\{x_3^* = 0\}$ and $\{x_4^* = 0\}$.
				\end{itemize}
				
				
				\item Each orbit of Type 2 is a part of a (transcendental or algebraic) hypersurface in $\G^*$ which is obtained 
				when ``cutting" it by two hyperplanes $\{x_4^* = 0\}$ and $\{x_5^* = 0\}$.
			\end{itemize}

			\item Now we consider in more detail one concrete group mentioned at the beginning of this item 
			(i.e. groups have two types of maximal-dimensional K-orbits), such as $G = G_4^{\lambda_1\lambda_2}$. 
			Denote by $\sim$ the equivalence relation which is defined on $\G^*$ as follows: $F_1 \sim F_2 \Leftrightarrow \Omega_{F_1} = \Omega_{F_2}$. 
			Then each K-orbit of $G$ can be seen as an element of the space $\G^*/\hspace{-5pt} \sim$ with quotient topology. 
			Assume~that
			\[
				F(\alpha_1,\alpha_2,\alpha_3, \alpha_4, \alpha_5,\alpha,\beta), F_0 (\alpha_1,\alpha_2,\alpha_3, 0, \alpha_5,\alpha,\beta) \in \G^*
			\]
			with $\alpha_i \neq 0$ for $i = 1, \ldots, 5$. Note that $\Omega_F$ is of Type 2 while $\Omega_{F_0}$ is of Type 1. 
			By letting $\alpha_4 \rightarrow 0$, we have
			\[
				F(\alpha_1,\alpha_2,\alpha_3, \alpha_4, \alpha_5,\alpha,\beta) \rightarrow F_0 (\alpha_1,\alpha_2,\alpha_3, 0,\alpha_5,\alpha,\beta).
			\]
			However, it is clear that the orbit $\Omega_F$ of $F$ does not converge to $\Omega_{F_0}$ in $\G^*/\hspace{-4pt}\sim$. 
			In other words, K-orbits of Type 1 seem to play a role as ``singularity'' and should be excluded from the family of K-orbits in general position. 
			Thus, the family of K-orbits in general position are only K-orbits of Type 2. For all remaining groups $G \in \A_1 \setminus \{G_4^{\lambda_1\lambda_2}\}$,
			we also have absolutely similar situations. In the following, we will formulate a definition for general position K-orbits.
		\end{enumerate}
		
		\item For each group $G \in \A_2$ 
		there is an infinite family of maximal-dimensional K-orbits. All these orbits are of the same type. 
		Namely, each of them is always a certain (transcendental or algebraic) hypersurfaces in $\G^*$. 
		Of course, all K-orbits of $G \in \A_2$
		are non-singular, and they are all in general~position.	
	\end{enumerate}
\end{rem}

Each maximal-dimensional K-orbit in the general position will be simply called a \emph{generic} K-orbit. Namely, we have the following definition.
		
\begin{defn}\label{D4}
Let $G \in \A$ ($= \A_1 \sqcup \A_2$). A six-dimensional K-orbit $\Omega$ of $G$ is called a \emph{generic K-orbit} if either $G \in \A_2$, or $G \in \A_1$  
and $\Omega$ is of Type 2.
\end{defn}

\begin{rem}[\textbf{Geometric characteristics of generic K-orbits}]\label{RemarkGenericOrbits}
	We will now examine more carefully the family $\F_G$ of all generic K-orbits as in Definition \ref{D4}. 
	For convenience, we set
	\[
	\begin{array}{l l}
	V_G := \cup \left\lbrace \Omega \, | \, \Omega \in \F_G \right\rbrace, & \text{for } G \in \A.
	\end{array}
	\] 
	By the picture of generic K-orbits of each group $G$, we can divide sixteen families of Lie groups in $\A$ into 
	three subfamilies such that each subfamily has an almost similar picture of generic K-orbits. 
	In the following, we will describe the geometric characteristics of the generic K-orbits for each subfamily.
	
	\begin{enumerate}
		\item For $G \in \A_1 \setminus \{G^\lambda_{12}\}$, 
		we have
		\[
		V_G = V_1: = \left\lbrace (x^*_1, \ldots, x^*_5, x^*, y^*) \in \G^* \, \big| \, x^*_4 x^*_5 \neq 0 \right\rbrace \subset \G^* \equiv \R^7.
		\]
		In fact,
		\begin{equation}\label{FoliatedManifold}
		V_1 \equiv \R^3 \times \left(\R \setminus \{0\}\right)^2 \times \R^2 = {V_1}_{++} \sqcup {V_1}_{-+} \sqcup {V_1}_{- -} \sqcup {V_1}_{+-}
		\end{equation}
		where
		\begin{equation}\label{SS1}
		{V_1}_{++} := \R^3 \times \R_+ \times \R_+ \times \R^2, \quad 
		{V_1}_{-+} := \R^3 \times \R_- \times \R_+ \times \R^2
		\end{equation}
		\begin{equation}\label{SS2}
		{V_1}_{--} := \R^3 \times \R_- \times \R_- \times \R^2, \quad 
		{V_1}_{+-} := \R^3 \times \R_+ \times \R_- \times \R^2
		\end{equation}
		in which $\R_+: =\{x \in \R: x > 0\}$ and $\R_-: = \{x \in \R :  x < 0\}$.
		
		\item If $G=G^\lambda_{12}$, we have
		\[
		V_G = V_2 := \left\lbrace (x^*_1, \dots, x^*_5, x^*, y^*) \in \G^* \, \big| \, x^*_5 \neq 0 \right\rbrace \subset \G^* \equiv \R^7.
		\]
		In fact,
		\begin{equation}\label{FoliatedManifold1}
		V_2 \equiv \R^4 \times \left(\R \setminus \{0\}\right) \times \R^2 = {V_2}_+ \sqcup {V_2}_- 
		\end{equation}
		where
		\begin{equation}\label{SS3}
		{V_2}_+ := \R^4 \times \R_{+} \times \R^2, \quad {V_2}_- : = \R^4 \times \R_- \times \R^2.
		\end{equation}
		
		\item For $G \in \A_2$, 
		we have
		\[
		V_G = V_3 := \left\lbrace (x^*_1, \dots, x^*_5, x^*, y^*) \in \G^* \, \big| \, {x_4^*}^2 + {x_5^*}^2 \neq 0 \right\rbrace \subset \G^* \equiv \R^7.
		\]
		In fact
		\begin{equation}\label{FoliatedManifold2} 
		V_3 \equiv \R^3 \times \left(\R^2 \setminus \{(0,0)\} \right) \times \R^2 \subset \G^* \equiv \R^7.  
		\end{equation}
	\end{enumerate}
	Obviously, all $V_G$ above are open submanifolds of $\G^* \equiv \R^7$ (with natural differential structure), 
	and each K-orbit $\Omega$ from $\F_G$ is a six-dimensional submanifold of $V_G$. 
	In the next section, we will prove that $\F_G$ forms a measurable foliation on the open submanifold $V_G$.
\end{rem}

{\small  \begin{landscape}
	\begin{longtable}{|p{.18\textwidth}| p{.125\textwidth}|p{.1\textwidth}| p{.5\textwidth}|p{.20\textwidth}|p{.10\textwidth}|}
		\caption{The picture of maximal-dimensional K-orbits}\label{tab2} \\
		\hline \multicolumn{2}{|c|}{Conditions for $\dim\Omega_F=6$} & $G$ & 
			$\Omega_F=\left\lbrace F_{U}(x_1^*, \dots, x_5^*, x^*, y^*)\right\rbrace$ & Geometric forms & Types \endfirsthead
		\caption*{Table 2 (continued)} \\
		\hline \multicolumn{2}{|c|}{Conditions for $\dim\Omega_F=6$} & $G$ & 
			$\Omega_F=\left\lbrace F_{U}(x_1^*, \dots, x_5^*, x^*, y^*)\right\rbrace$ & Geometric forms & Types \\	 
		\hline \endhead 
		
		\hline \multirow{4}{*}{$\begin{array}{l} \alpha_5 \neq 0 \\ \alpha_2^2+\alpha_4^2 \neq 0 \\ \alpha_1, \alpha_3, \alpha, \beta \in \R \end{array}$} &
			$\alpha_2 \alpha_5 \neq 0 = \alpha_4$ & \multirow{4}{*}{$G^\lambda_1$} & 
			$x_4^*=0$, $\alpha_2 x^*_2>0$, $\alpha_5 x^*_5 >0$, $x_1^*, x_3^*, x^*, y^* \in \R$ & 
			a part of hyperplane $\{x^*_4 = 0\}$ & singularity \\
		\cline{2-2} \cline{4-6} & $\alpha_4 \alpha_5 \neq 0$ & & $x^*_4=\alpha_4$, $\alpha_5x^*_5 >0$, $x_1^*,x_2^*,  x_3^*, x^*, y^* \in \R$ & 
			a part of hyperplane $\{x^*_4 = \alpha_{4}\}$ & generic \\
		
		\hline \multirow{7}{*}{$\begin{array}{l} \alpha_4 \neq 0, \\ \alpha_3^2+\alpha_5^2 \neq 0, \\ \alpha_1,\alpha_2, \alpha, \beta \in \R \end{array}$} &
			$\alpha_3 \alpha_4 \neq 0 = \alpha_5$ & $G_2$, $G_{11}$ & $x_5^*=0$, $\alpha_3 x^*_3>0$, $\alpha_4 x^*_4 >0$, $x_1^*, x_2^*, x^*, y^* \in \R$ &
			a part of hyperplane $\{x^*_5 = 0\}$ & singularity \\
		\cline{2-6} & \multirow{4}{*}{$\alpha_4 \alpha_5 \neq 0$} & $G_2$ & 
			$x^*_2-\frac{x^*_3 x^*_4}{x^*_5} = \alpha_2-\frac{\alpha_3\alpha_4}{\alpha_5}$,
				$\alpha_4 x^*_4>0, \alpha_5 x^*_5 >0, x_1^*, x^*, y^* \in \R$ & 
			a part of an algebraic hypersurface & generic \\
		\cline{3-6} & & $G_{11}$ & 
			$\left(\frac{x^*_2}{x_5}-\frac{x^*_3x^*_4}{x^{*2}_5} \right) e^{\frac{x^*_4}{x^*_5}} 
				= \left(\frac{\alpha_2}{\alpha_5}-\frac{\alpha_3\alpha_4}{\alpha^2_5} \right) e^{\frac{\alpha_4}{\alpha_5}}$,
			$\alpha_4 x^*_4>0$, $\alpha_5 x^*_5 >0, x_1^*, x^*, y^* \in \R$ &
			a part of a transcendental hypersurface & generic \\
		
		\hline \multirow{10}{*}{$\begin{array}{l} \alpha_4^2+\alpha_5^2 \neq 0, \\ \text{except for} \\ \alpha_3=\alpha_5=0 \neq \alpha_4, \\
			\alpha_1,\alpha_2, \alpha, \beta \in \R \end{array}$} & $\alpha_4 =0\neq\alpha_5$ & \multirow{4}{*}{$G_7$, $G_8$} &
			$x^*_4 = 0$, $\alpha_5 x^*_5 >0$, $x_1^*, x_2^*, x_3^*, x^*, y^* \in \R$ & a part of hyperplane $\{x^*_4 = 0\}$ & singularity \\
		\cline{2-2} \cline{4-6} & $\alpha_3\alpha_4\neq0=\alpha_5$ & & 
			$x_5^*=0$, $\alpha_3 x^*_3>0$, $\alpha_4 x^*_4 >0$, $x_1^*, x_2^*, x^*, y^* \in \R$ & a part of hyperplane $\{x^*_5 = 0\}$ & singularity \\
		\cline{2-6} & \multirow{7}{*}{$\alpha_4\alpha_5\neq0$} & \multirow{3}{*}{$G_7$} & 
			$\frac{x^*_2}{x^*_5} - \frac{x^*_3 x^*_4}{{x^*_5}^2}-\ln \vert x^*_4\vert + \ln \vert x^*_5 \vert
			=\frac{\alpha_2}{\alpha_5} - \frac{\alpha_3\alpha_4}{\alpha^2_5} - \ln\vert\alpha_4\vert +  \ln\vert\alpha_5\vert$,
			$\alpha_4 x^*_4 > 0$, $\alpha_5 x^*_5 >0$, $x_1^*,  x^*, y^* \in \R$ & a part of a transcendental hypersurface & generic \\
		\cline{3-6} & & \multirow{4}{*}{$G_8^\lambda$} & 
			$\frac{x^*_2}{x^*_5} - \frac{x^*_3 x^*_4}{{x^*_5}^2} -(2+\lambda)\ln\vert x^*_5\vert 
			+ (1+\lambda)\ln\vert x^*_4\vert = \frac{\alpha_2}{\alpha_5} -\frac{\alpha_3\alpha_4}{\alpha^2_5} 
			- (2+\lambda)\ln\vert\alpha_5\vert + (1+\lambda)\ln\vert\alpha_4\vert$, 
			$\alpha_4 x^*_4 > 0$, $\alpha_5 x^*_5 >0$, $x_1^*,  x^*, y^* \in \R$ & a part of a transcendental hypersurface & generic \\
		
		\hline \multirow{18}{*}{$\begin{array}{l}\alpha_4 = 0 \neq \alpha_2\alpha_5, \text{or } \\ \alpha_4 \neq 0 \neq \alpha_3^2+\alpha_5^2,\\
			\alpha_1, \alpha, \beta \in \R \end{array}$} & $\alpha_4 = 0 \neq \alpha_2\alpha_5$ & 
			$G_3$, $G_4^{\lambda_1\lambda_2}$, $G_5$, $G_6^\lambda$, $G_9$, $G_{10}^\lambda$
			& $x_4^*=0$, $\alpha_2 x^*_2>0$, $\alpha_5 x^*_5 >0$, $x_1^*, x_3^*, x^*, y^* \in \R$ & a part of hyperplane $\{x^*_4 = 0\}$ & singularity \\
		\cline{2-2} \cline{4-6} & $\alpha_5 = 0 \neq \alpha_3\alpha_4$ & & $x_5^*=0$, $\alpha_3 x^*_3>0$, $\alpha_4 x^*_4 >0$, $x_1^*, x_2^*, x^*, y^* \in \R$
			&  part of hyperplane $\{x^*_5 = 0\}$ & singularity \\
		\cline{2-6} & & $G_3$ & $\frac{x^*_2}{x^*_4} - \frac{x^*_3}{x^*_5} = \frac{\alpha_2}{\alpha_4} - \frac{\alpha_3}{\alpha_5}$, 
			$\alpha_4 x^*_4>0$, $\alpha_5 x^*_5 >0$, $x_1^*, x^*, y^* \in \R$ & a part of an algebraic hypersurface & generic  \\
		\cline{3-6} & & $G_4^{\lambda_1\lambda_2}$ & $\left(x^*_2-\frac{x^*_3x^*_4}{x^*_5}\right)
			\frac{{x^*_4}^{\frac{1+\lambda_1}{\lambda_2-\lambda_1-1}}}{{x^*_5}^{\frac{1}{\lambda_2-\lambda_1-1}}} 
			= \left( \alpha_2-\frac{\alpha_3\alpha_4}{\alpha_5}\right) 
			\frac{{\alpha_4}^{\frac{1+\lambda_1}{\lambda_2-\lambda_1-1}}}{{\alpha_5}^{\frac{1}{\lambda_2-\lambda_1-1}}}$,
			$\alpha_4 x^*_4 >0$, $\alpha_5 x^*_5 >0$, $x_1^*, x^*, y^* \in \R$&a part of an algebraic hypersurface& generic\\
		\cline{3-6} & $\alpha_4 \alpha_5 \neq 0$ & $G_5, G_6^\lambda$ & 
			$\left(\frac{x^*_2}{x^*_4}-\frac{x^*_3}{x^*_5}\right) \sqrt{\vert x^*_4 \vert} = \left(\frac{\alpha_2}{\alpha_4}-\frac{\alpha_3}{\alpha_5}\right) 
			\sqrt{\vert\alpha_4 \vert }$, $\alpha_4 x^*_4 >0$, $\alpha_5 x^*_5 >0$, $x_1^*, x^*, y^* \in \R$ & a part of an algebraic hypersurface & generic \\
		\cline{3-6} & & $G_9$ & $\frac{x^*_2}{x^*_4} - \frac{x^*_3}{x^*_5} + \ln\vert x^*_4 \vert = 
			\frac{\alpha_2}{\alpha_4} - \dfrac{\alpha_3}{\alpha_5} + \ln\vert \alpha_4 \vert$, $\alpha_4 x^*_4 >0$, $\alpha_5 x^*_5 >0$, $x_1^*, x^*, y^* \in \R$
			& a part of a transcendental hypersurface & generic \\
		\cline{3-6} & & $G_{10}^\lambda$ & $\frac{x^*_2}{x^*_4} - \frac{x^*_3}{x^*_5} -\lambda\ln\vert x^*_4 \vert + \ln\vert x^*_5 \vert
			= \frac{\alpha_2}{\alpha_4} - \frac{\alpha_3}{\alpha_5} -\lambda\ln\vert \alpha_4 \vert + \ln\vert \alpha_5 \vert$, 
			$\alpha_4 x^*_4>0$, $\alpha_5 x^*_5 >0$, $x_1^*, x^*, y^* \in \R$ & a part of a transcendental hypersurface & generic \\	
		
		\hline \multirow{6}{*}{$\begin{array}{l}\alpha_3 \alpha_4 \neq 0 = \alpha_5, \\ \text{or } \alpha_5\neq 0,\\ \alpha_1,\alpha_2,  \alpha, \beta \in \R\end{array}$}	
			& $\alpha_3 \alpha_4 \neq 0 = \alpha_5$ & \multirow{6}{*}{$G_{12}^\lambda$} & 
			$x^*_5 = 0$, $\alpha_3 x^*_3>0$, $\alpha_4 x^*_4 >0$, $x_1^*, x_2^*, x^*, y^* \in \R$ & a part of the hyperplane $\{x^*_5 = 0\}$ & singularity \\
		\cline{2-2} \cline{4-6} & $\alpha_5\neq0$ & & $\left(x^*_2-\frac{x^*_3x^*_4}{x^*_5}\right)\frac{1}{{x^*_5}^{\frac{\lambda}{1+\lambda}} 
			e^{\frac{x^*_4}{(1+\lambda)x^*_5}}} = \left(\alpha_2-\frac{\alpha_3\alpha_4}{\alpha_5}\right)
			\frac{1}{\alpha_5^{\frac{\lambda}{1+\lambda}}e^{\frac{\alpha_4}{(1+\lambda)\alpha_5}}}$,
			$\alpha_5 x^*_5 >0$, $x_1^*, x^*, y^* \in \R$ & a half of a transcendental hypersurface & generic \\		
		
		\hline \multicolumn{2}{|c|}{} & \multirow{4}{*}{$G^\lambda_{13}$} & $\frac{x^*_2x^*_5-x^*_3x^*_4}{{x^{*}_4}^2+{x^{*}_5}^2}
			e^{\lambda \arctan\frac{x^*_4}{x^*_5}}
			= \frac{\alpha_2\alpha_5-\alpha_3\alpha_4}{{\alpha_4}^2+{\alpha_5}^2}e^{\lambda\arctan\frac{\alpha_4}{\alpha_5}}$,
			$x_1^*, x^*, y^* \in \R$, $x_4^*x_5^* \neq 0$ & &	 \\
		\cline{4-4} \multicolumn{2}{|c|}{} & & $\frac{x^*_2}{x_5^*} = \frac{\alpha_2\alpha_5-\alpha_3\alpha_4}{{\alpha_4}^2+{\alpha_5}^2}
			e^{\lambda\arctan\frac{\alpha_4}{\alpha_5}}$, $x_1^*, x^*, y^* \in \R$, $x_4^*=0 \neq x_5^*$ & & \\
		\cline{3-4} \multicolumn{2}{|c|}{$\begin{array}{l}	\alpha_4^2+\alpha_5^2 \neq 0, \\ \alpha_1,\alpha_2, \alpha_3, \alpha, \beta \in \R \end{array}$} 
			& $G_{14}^{\lambda_1\lambda_2}$ & 
			$\frac{x^*_3x^*_4-x^*_2x^*_5}{\left({x_4^*}^2 +{x_5^*}^2\right)^{\frac{2\lambda_1+1}{2(1+\lambda_1)}}}
			e^{\frac{\lambda_2}{1+\lambda_1}\arctan\frac{x^*_5}{x^*_4}} =
			\frac{\alpha_3\alpha_4-\alpha_2\alpha_5}{\left(\alpha_4^2+\alpha_5^2 \right)^{\frac{2\lambda_1+1}{2(1+\lambda_1)}}}
			e^{\frac{\lambda_2}{1+\lambda_1}\arctan\frac{\alpha_5}{\alpha_4}}$, $x_1^*, x^*, y^* \in \R$, $x_4^*x_5^* \neq 0$ 
			&  &  \\ 
			\cline{4-4} \multicolumn{2}{|c|}{} & & $\frac{x^*_3}{{x^*}_4^{\frac{\lambda_1}{1+\lambda_1}} }
			=
			\frac{\alpha_3\alpha_4-\alpha_2\alpha_5}{\left(\alpha_4^2+\alpha_5^2 \right)^{\frac{2\lambda_1+1}{2(1+\lambda_1)}}}e^{\frac{\lambda_2}{1+\lambda_1}\arctan\frac{\alpha_5}{\alpha_4}}$, $x_1^*, x^*, y^* \in \R$, $x_4^* \neq 0 = x_5^*$ & transcendental hypersurface&generic \\
		\cline{3-4} \multicolumn{2}{|c|}{} & \multirow{4}{*}{$G_{15}$} & $\frac{x^*_3x^*_4-x^*_2x^*_5}{{x_4^*}^2 + {x_5^*}^2}+\frac{\ln({x_4^*}^2+{x_5^*}^2)}{2}
			= \frac{ \alpha_3\alpha_4-\alpha_2\alpha_5}{\alpha_4^2+\alpha_5^2}+ \frac{\ln (\alpha_4^2+\alpha_5^2)}{2}$,
			$x_1^*, x^*, y^* \in \R$, $x_4^* x_5^* \neq 0$ & & \\
		\cline{4-4} \multicolumn{2}{|c|}{} & & $-\frac{x^*_2}{x^*_5}+\ln \vert x^*_5\vert = \frac{\alpha_3\alpha_4-\alpha_2\alpha_5}{\alpha_4^2+\alpha_5^2}
			+ \frac{\ln (\alpha_4^2+\alpha_5^2)}{2}$, $x_1^*, x^*, y^* \in \R$, $x_4^*=0 \neq x_5^*$ & & \\ 
		\cline{3-4} \multicolumn{2}{|c|}{} & \multirow{4}{*}{$G_{16}^\lambda$} & $\frac{x^*_3x^*_4-x^*_2x^*_5}{{x_4^*}^2+{x_5^*}^2} +
			\frac{x^*_5x^*_4}{2\left({x_4^*}^2 + {x_5^*}^2\right)} + \frac{\lambda\ln\left({x^*}^2_4+{x^*}^2_5 \right)}{2}
			- \frac{\arctan\frac{x^*_5}{x^*_4}}{2} = \frac{\alpha_3\alpha_4-\alpha_2\alpha_5}{\alpha^2_4+\alpha_5^2} +
			\frac{\alpha_5\alpha_4}{2(\alpha^2_4+\alpha_5^2)} + \frac{\lambda \ln \left({\alpha}^2_4+{\alpha}^2_5 \right) - \arctan\frac{\alpha_5}{\alpha_4}}{2},$
			$x_1^*, x^*, y^* \in \R$, $x_4^* x_5^* \neq 0$ &  &  \\
		\cline{4-4} \multicolumn{2}{|c|}{} & & $\frac{x^*_3}{x^*_4} +  \lambda \ln \vert x^*_4 \vert = \frac{\alpha_3\alpha_4-\alpha_2\alpha_5}{\alpha^2_4+\alpha_5^2} 
			+ \frac{\alpha_5\alpha_4}{2(\alpha^2_4+\alpha_5^2)} + \frac{\lambda \ln \left({\alpha}^2_4+{\alpha}^2_5 \right)}{2} 
			- \frac{\arctan\frac{\alpha_5}{\alpha_4}}{2}$, $x_1^*, x^*, y^* \in \R$, $x_4^* \neq 0 = x_5^*$ & & \\ \hline
	\end{longtable}
\end{landscape}}

\section{Foliations formed by generic K-orbits of considered Lie groups}\label{sec3}

This section establishes the remaining results of the paper on foliations formed by generic K-orbits of each Lie group $G \in \A$. 
Recall that all $V_G \in \{V_1, V_2, V_3\}$ determining by \eqref{FoliatedManifold}, \eqref{FoliatedManifold1}, \eqref{FoliatedManifold2}
are open submanifolds of $\G^* \equiv \R^7$, and $\F_G$ is the family of all generic K-orbits of $G$ (see Definition \ref{D4}).

\begin{thm}\label{FormedFoliation}
	For each group $G \in \A$, the family $\F_G$ of all generic K-orbits of $G$ forms a measurable foliation on the open manifold $V_G$ 
	in the sense of Connes \cite{con}, and it is called the foliation associated with $G$.
\end{thm}

\begin{proof}
	The proof is analogous to the case of MD-groups in \cite{vu90, v-t, v-h, v-h-t} and Lie groups in \cite[Theorem 21]{t-v}. 
	It consists of two steps as follows
	\begin{itemize}
		\item Step 1: For any $G \in \A$, construct a suitable differential system $S_G$ of rank six on the manifold $V_G$ 
		such that each K-orbit $\Omega$ from $\F_G$ is a maximal connected integral submanifold corresponding to $S_G$.
		\item Step 2: Afterwards, show that the Lebesgue measure is invariant for some smooth polyvector field $\X$ 
		of degree six such that it generates $S_G$.
	\end{itemize}
	 
	As emphasized in Remark \ref{RemarkGenericOrbits}, sixteen families of Lie groups in $\A$ can be divided into three subfamilies as follows
	\[
		\A = \left\lbrace G^{\lambda}_{1}, G_{2}, G_{3}, G_{4}^{\lambda_1\lambda_2}, G_5, G_{6}^{\lambda}, G_7, G_{8}^{\lambda}, 
		G_{9}, G^{\lambda}_{10}, G_{11} \right\rbrace \cup \left\lbrace G_{12}^{\lambda} \right\rbrace \cup 
		\left\lbrace G^{\lambda}_{13}, G_{14}^{\lambda_1\lambda_2}, G_{15}, G_{16}^{\lambda} \right\rbrace,
	\]
	where groups in the same subfamily have almost the same picture of generic K-orbits. To prove Theorem \ref{FormedFoliation}, 
	we thus choose in each of these subfamilies a representative group and calculate carefully for that selected group.
	
	Note that the assertion of Theorem \ref{FormedFoliation} for cases $G_2$, $G_3$, $G_9$ and $G^\lambda_{10}$ 
	had been proved in \cite{t-v}. In the following, we choose three groups $G_4^{\lambda_1\lambda_2}$, $G_{12}^{\lambda}$ and $G_{13}^{\lambda}$
	and prove Theorem \ref{FormedFoliation} explicitly for this cases.
	
	Let $G=G_4^{\lambda_1\lambda_2}$. {\it Proof of Step 1.} For any $F(\alpha_1,\alpha_2, \alpha_3, \alpha_4, ,\alpha_5,\alpha,\beta) \in \G^*$,
	by Theorem \ref{picture K-orbit}, the generic K-orbit $\Omega_F$ belongs to $\F_G$ if and only if $\alpha_4\alpha_5 \neq 0$. 
	Furthermore, if we denote by $v$ the element $(x^*_1, x^*_2, x^*_3, x^*_4, x^*_5, x^*, y^*) \in \G^*$ then
	\begin{flalign*}
		\Omega_F = \Bigg\{ v \in \G^* \, \bigg| \left( x^*_2-\frac{x^*_3x^*_4}{x^*_5}\right) 
			\frac{{x^*_4}^{\frac{1+\lambda_1}{\lambda_2-\lambda_1-1}}}{{x^*_5}^{\frac{1}{\lambda_2-\lambda_1-1}}} 
			= \left( \alpha_2-\frac{\alpha_3\alpha_4}{\alpha_5}\right) \frac{{\alpha_4}^{\frac{1+\lambda_1}{\lambda_2-\lambda_1-1}}}
				{{\alpha_5}^{\frac{1}{\lambda_2-\lambda_1-1}}},
			\alpha_4 x^*_4 > 0, \, \alpha_5 x^*_5 >0 \Bigg\}.
	\end{flalign*}
	
\noindent Now we consider on $V_G$ the following differential system $S_G$: 
	\[
		\begin{cases}
			\X_1 :=\frac{\partial}{\partial x^*_1} \\
			\X_2 :=  \lambda_1x^*_3\frac{\partial}{\partial x^*_3} + x^*_4\frac{\partial}{\partial x^*_4} + (\lambda_1 + 1)x^*_5\frac{\partial}{\partial x^*_5} \\
			\X_3 :=x^*_2\frac{\partial}{\partial x^*_2} + \lambda_2x^*_3\frac{\partial}{\partial x^*_3} + x^*_4\frac{\partial}{\partial x^*_4} 
						+ \lambda_2x^*_5\frac{\partial}{\partial  x^*_5} \\
			\X_4 :=x^*_4\frac{\partial}{\partial x^*_2} +x^*_5\frac{\partial}{\partial x^*_3} \\
			\X_5:= \frac{\partial}{\partial x^*} \\
			\X_6:=\frac{\partial}{\partial y^*}.
		\end{cases}
	\]
	Obviously, $\Rank(S_G)=6$ and all $\X_i \,(i= 1, \ldots, 6)$ are smooth over $V_G$. 
	We claim that $S_G$ generates $\F_G$, i.e. each K-orbit $\Omega$ from $\F_G$ is a maximal connected integral submanifold of $S_G$.
	
\noindent	If fact, we first consider $\X_1, \X_5, \X_6$. Clearly, their flows (i.e. one-parameter subgroups) are determined as follows
	\[
		\begin{array}{l l}
			\theta^{\X_1}_{x^*_1-\alpha_1}: & F \mapsto \theta^{\X_1}_{x^*_1-\alpha_1}(F):=(x^*_1, \alpha_2,\alpha_3, \alpha_4, \alpha_5, \alpha, \beta), \\
			\theta^{\X_5}_{x^*-\alpha}: & F \mapsto \theta^{\X_5}_{x^*-\alpha}(F):=(\alpha_1, \alpha_2,\alpha_3, \alpha_4, \alpha_5, x^*, \beta), \\
			\theta^{\X_6}_{y^*-\beta}: & F \mapsto \theta^{\X_6}_{y^*-\beta}(F):=(\alpha_1, \alpha_2,\alpha_3, \alpha_4, \alpha_5, \alpha, y^*).
			\end{array}
	\]
	
\noindent	Next, we consider $\X_2$. For some positive $\epsilon \in \R$, assume that
	\[
		\varphi: t \mapsto \varphi(t) = \left(x^*_1(t), x^*_2(t), x^*_3(t), x^*_4(t), x^*_5(t), x^*(t), y^*(t) \right), \quad t \in (-\epsilon, \epsilon)
	\]
	is an integral curve of $\X_2$ passing $F=\varphi(0)$. Then, we must have $\varphi'(t)={\X_2}_{\varphi(t)}$ and it is equivalent to
	\[
		\begin{array}{l}
			\sum \limits_{i=1}^{5}x^{*'}_i(t)\frac{\partial}{\partial x^*_i}+x^{*'}(t)\frac{\partial}{\partial x^*}+y^{*'}(t)\frac{\partial}{\partial y^*} 
			= \lambda_1x^*_3(t)\frac{\partial}{\partial  x^*_3} + x^*_4(t)\frac{\partial}{\partial  x^*_4} + (1+\lambda_1)x^*_5(t)\frac{\partial}{\partial  x^*_5}\cdot
		\end{array}
	\]
	By identifying coefficients on two sides, we obtain
	\begin{equation}\label{G3X2}
		\begin{cases}
			x^{*'}_1(t)= x^{*'}_2(t)=x^{*'}(t)= y^{*'}(t)=0 \\
			x^{*'}_3(t)=\lambda_1x^{*}_3(t) \\
			x^{*'}_4(t)=x^*_4(t) \\
			x^{*'}_5(t)=(1+\lambda_1)x^*_5(t).
		\end{cases}
	\end{equation}
	Since $F=\varphi(0)$, system \eqref{G3X2} gives us
	\[
		x^*_1=\alpha_1,\, x^*_2=\alpha_2,\, x^*_3= \alpha_3e^{\lambda_1t},\, x^*_4=\alpha_4e^t,\, x^*_5=\alpha_5e^{(1+\lambda_1)t},\, x^*=\alpha,\, y^*=\beta.
	\]
	Therefore, the flow of $\mathfrak{X}_2$ is as follows
	\[
		\theta^{\X_2}_{x}: F \mapsto \theta^{\X_2}_{x}(F):=(\alpha_1, \alpha_2,\alpha_3e^{\lambda_1x}, \alpha_4e^x, \alpha_5e^{(1+\lambda_1)x}, \alpha, \beta).
	\]	
	Similarly, the flows of $\X_3$ is as follows
	\[
		\theta^{\X_3}_{y}: F \mapsto \theta^{\X_3}_{y}(F):=(\alpha_1, \alpha_2e^y,\alpha_3e^{\lambda_2y}, \alpha_4e^y, \alpha_5e^{\lambda_2y}, \alpha, \beta).
	\]	
	Now, we consider $\X_4$. Assume that
	\[
		\varphi: t \mapsto \varphi(t)=(x^*_1(t), x^*_2(t), x^*_3(t), x^*_4(t), x^*_5(t), x^*(t), y^*(t))
	\]
	be an integral curve of $\X_4$ passing $F=\varphi(0)$, where $t \in (-\epsilon, \epsilon) \subset \R$ for some positive $\epsilon \in \R$. 
	Then we have $\varphi'(t)={\X_4}_{\varphi(t)}$ which is equivalent to
	\begin{align}\label{G3x4}
		& \sum_{i=1}^{5}x^{*'}_i(t)\frac{\partial}{\partial x^*_i}+x^{*'}(t)\frac{\partial}{\partial x^*}+y^{*'}(t)\frac{\partial}{\partial y^*}
			=x^*_4(t) \frac{\partial}{\partial  x^*_2} +x^*_5(t)\frac{\partial}{\partial  x^*_3} \nonumber\\
		\Leftrightarrow &
		\begin{cases}
			x^{*'}_1(t)= x^{*'}_4(t)=
			x^{*'}_5(t)= x^{*'}(t)= y^{*'}(t)=0\\
			x^{*'}_2(t)=x^*_4(t)\\
			x^{*'}_3(t)=x^*_5(t). 
		\end{cases} 
	\end{align}
	Since $F=\varphi(0)$, equation \eqref{G3x4} gives us
	\[
		x^*_1=\alpha_1,\,x^*_2= \alpha_2+\alpha_4t,\, x^*_3=\alpha_3+\alpha_5t,\,x^*_4=\alpha_4,\, x^*_5=\alpha_5,\, x^*=\alpha,\; y^*=\beta.
	\]
	Therefore, the flow of $\X_4$ is as follows
	\[
		\theta^{\X_4}_{x_1}: F \mapsto \theta^{\X_2}_{x_1}(F):=(\alpha_1,\alpha_2+\alpha_4x_1,\alpha_3+\alpha_5x_1,\alpha_4, \alpha_5, \alpha, \beta).
	\]
	
	Finally, we set $\theta=\theta^{\X_6}_{y^*-\beta} \circ \theta^{\X_5}_{x^*-\alpha} \circ \theta^{\X_4}_{x_1} \circ 
	\theta^{\X_3}_y \circ \theta^{\X_2}_{x} \circ\theta^{\X_1}_{x^*_1-\alpha_1}$. Then
	\begin{flalign*}
		\theta(F) & = \theta^{\X_6}_{y^*-\beta} \circ \theta^{\X_5}_{x^*-\alpha} \circ \theta^{\X_4}_{x_1} 
				\circ \theta^{\X_3}_y \circ \theta^{\X_2}_x \circ \theta^{\X_1}_{x^*_1-\alpha_1}(F) \\
		& = (x^*_1, x^*_2, x^*_3, x^*_4, x^*_5, x^*, y^*)
	\end{flalign*}
	where $x^*_1, x^*, y^* \in \R$ and
	\[
		\begin{array}{l l}
			x^*_2= \alpha_2e^{y} + \alpha_4x_1e^{y}, & x^*_3= \alpha_3 e^{\lambda_1x+\lambda_2y} + \alpha_5x_1e^{\lambda_1x+\lambda_2y}, \\
			x^*_4= \alpha_4e^{x+y}, & x^*_5=\alpha_5 e^{\lambda_1x+x+\lambda_2y}.
		\end{array}	
	\]
	By direct calculations, we get 
	\[
		\begin{array}{l}
				x^*_2 = \frac{x^*_3x^*_4}{x^*_5} + \left( \alpha_2-\frac{\alpha_3\alpha_4}{\alpha_5}\right) 
					\frac{{\alpha_4}^{\frac{1+\lambda_1}{\lambda_2-\lambda_1-1}}}{{\alpha_5}^{\frac{1}{\lambda_2-\lambda_1-1}}}
					\frac{{x^*_5}^{\frac{1}{\lambda_2-\lambda_1-1}}}{{x^*_4}^{\frac{1+\lambda_1}{\lambda_2-\lambda_1-1}}} 
			\Leftrightarrow  \left( x^*_2-\frac{x^*_3x^*_4}{x^*_5}\right) 
					\frac{{x^*_4}^{\frac{1+\lambda_1}{\lambda_2-\lambda_1-1}}}{{x^*_5}^{\frac{1}{\lambda_2-\lambda_1-1}}} 
					= \left( \alpha_2-\frac{\alpha_3\alpha_4}{\alpha_5}\right) 
					\frac{{\alpha_4}^{\frac{1+\lambda_1}{\lambda_2-\lambda_1-1}}}{{\alpha_5}^{\frac{1}{\lambda_2-\lambda_1-1}}}\cdot
		\end{array}
	\]
	Hence, $\left\lbrace \theta(F) \, | \, x^*_1, x^*_3, x^*_4, x^*_5, x^*, y^* \in \R, \alpha_4x^*_4>0, \alpha_5x^*_5>0 \right\rbrace \equiv \Omega_F$,
	i.e. $\Omega_F$ is a maximal connected integral submanifold corresponding to $S_{G}$. 
	In other words, $S_G$ generates $\F_G$ and $(V_G, \F_G)$ is a six-dimensional foliation for $G = G_4^{\lambda_1\lambda_2}$, 
	where $V_G = V_1$ as in \eqref{FoliatedManifold}.
	
	\noindent{\it Proof of Step 2.} We have to show that the foliation $(V_G, \F_G)$ is measurable in the sense of Connes. 
	As mentioned in \cite[Subsection 2.2]{t-v}, to prove that $(V_G, \F_G)$ is measurable, 
	we only need to choose some suitable pair ($\X, \mu$) on $V_G$, where $\X$ is some smooth 6-vector field defined on $V_G$,
	$\mu$ is some measure on $V_G$ such that $\X$ generates $S_G$ and $\mu$ is $\X$-invariant.
	Namely, we choose $\mu$ to be exactly the Lebesgue measure on $V_G$ and set 
	$\X := \X_1 \wedge \X_2 \wedge \X_3 \wedge \X_4 \wedge \X_5 \wedge \X_6$.
	Clearly, $\X$ is smooth, non-zero everywhere on $V_G$ and it is exactly a polyvector field of degree six. 
	Moreover, $\X$ generates $S_G$. In other words, if we choose a suitable orientation on $(V_G, \F_G)$ then 
	$\X \in C^\infty \left(\Lambda^6(\F) \right)^+$. Besides, the invariance of the Lebesgue measure $\mu$ with respect to $\X$ 
	is equivalent to the invariance of $\mu$ for the K-representation that is restricted to the foliated submanifold $V_G$ in $\G^*$. 
	For any $U(x_1, x_2, x_3, x_4, x_5, x, y) \in \G$, direct computations show that the Jacobi determinant
	$J_U$ of differential map $K \left(\exp_G(U)\right)$ is a constant which depends only on $U$ but does not depend on 
	the coordinates of any point which moves in each generic K-orbit $\Omega \in \F_G$. 
	This means that the Lebegue measure $\mu$ is $\X$-invariant.
	This completes the proof for the case $G = G_4^{\lambda_1\lambda_2}$.

Next, we consider the case $G = G_{12}^{\lambda}$. {\it Proof of Step 1.} For any $F(\alpha_1,\alpha_2, \alpha_3, \alpha_4, ,\alpha_5,\alpha,\beta) \in \mathcal{G}^*$,
by Theorem \ref{picture K-orbit}, the generic K-orbits $\Omega_F$ belongs to $\mathcal{F}_G$ if and only if $\alpha_5 \neq 0$. 
Furthermore, if we denote by $v$ the element $(x^*_1, x^*_2, x^*_3, x^*_4, x^*_5, x^*, y^*) \in \G^*$ then
\[\begin{array}{r}
\Omega_F=\Bigg\{v \in \mathcal{G}^* \, \bigg| \, \left(x^*_2-\frac{x^*_3x^*_4}{x^*_5}\right)\frac{1}{{x^*_5}^{\frac{\lambda}{1+\lambda}}e^{\frac{x^*_4}{(1+\lambda)x^*_5}}}  =\left(\alpha_2-\frac{\alpha_3\alpha_4}{\alpha_5}\right)\frac{1}{\alpha_5^{\frac{\lambda}{1+\lambda}}e^{\frac{\alpha_4}{(1+\lambda)\alpha_5}}} , 
\;  \alpha_5 x^*_5 >0 \Bigg\}.\end{array}\]
On the open submanifold $V_G$ we consider the following differential system $S_{G}$: 
\[\begin{cases}
\mathfrak{X}_1:=\frac{\partial}{\partial x^*_1}
\\
\mathfrak{X}_2:=  \lambda
x^*_2\frac{\partial}{\partial  x^*_2}  +	\lambda x^*_3\frac{\partial}{\partial  x^*_3} + (\lambda+1)x^*_4\frac{\partial}{\partial  x^*_4} + (\lambda+1)x^*_5\frac{\partial}{\partial  x^*_5}
\\
\mathfrak{X}_3:=(x^*_2+x^*_3)\frac{\partial}{\partial  x^*_2} + x^*_3\frac{\partial}{\partial  x^*_3} + (x^*_4+x^*_5)\frac{\partial}{\partial  x^*_4} +  x^*_5 \frac{\partial}{\partial  x^*_5}
\\
\mathfrak{X}_4:=x^*_4\frac{\partial}{\partial  x^*_2} +x^*_5\frac{\partial}{\partial  x^*_3}  
\\
\mathfrak{X}_5:= \frac{\partial}{\partial x^*} 
\\
\mathfrak{X}_6:=\frac{\partial}{\partial y^*} \cdot
\end{cases}
\]
Obviously, $\Rank(S_G)=6$ and all $\mathfrak{X}_i \,(i= 1, \ldots, 6)$ are smooth over $V_G$. Now, we will show that $S_{G}$ generates $\mathcal{F}_{G}$.

\noindent First, we consider 
$\mathfrak{X}_1, \mathfrak{X}_5, \mathfrak{X}_6$. Their flows (i.e. one-parameter subgroups) are determined as follows
\[	\begin{array}{ll}
\theta^{\mathfrak{X}_1}_{x^*_1-\alpha_1}:& F \mapsto \theta^{\mathfrak{X}_1}_{x^*_1-\alpha_1}(F):=(x^*_1, \alpha_2,\alpha_3, \alpha_4, \alpha_5, \alpha, \beta)\\
\theta^{\mathfrak{X}_5}_{x^*-\alpha}:& F \mapsto \theta^{\mathfrak{X}_5}_{x^*-\alpha}(F):=(\alpha_1, \alpha_2,\alpha_3, \alpha_4, \alpha_5, x^*, \beta)\\
\theta^{\mathfrak{X}_6}_{y^*-\beta}:& F \mapsto \theta^{\mathfrak{X}_6}_{y^*-\beta}(F):=(\alpha_1, \alpha_2,\alpha_3, \alpha_4, \alpha_5, \alpha, y^*).
\end{array}
\]
Next, we consider $\mathfrak{X}_2:= \lambda
x^*_2\frac{\partial}{\partial  x^*_2}  +	\lambda x^*_3\frac{\partial}{\partial  x^*_3} + (\lambda+1)x^*_4\frac{\partial}{\partial  x^*_4} + (\lambda+1)x^*_5\frac{\partial}{\partial  x^*_5}$. For some positive $\epsilon \in {\mathbb R}$, assume that
\[\varphi: t \mapsto \varphi(t)=\bigl(x^*_1(t), x^*_2(t), x^*_3(t), x^*_4(t), x^*_5(t), x^*(t), y^*(t)\bigr); \;  t \in (-\epsilon, \epsilon)\]
is an integral curve of $\X_2$ passing $F=\varphi(0)$.
Then, $\varphi'(t)={\X_2}_{\varphi(t)}$, that is
\begin{align}\label{G3X2}
& \sum_{i=1}^{5}x^{*'}_i(t)\frac{\partial}{\partial x^*_i}+x^{*'}(t)\frac{\partial}{\partial x^*}+y^{*'}(t)\frac{\partial}{\partial y^*}= \nonumber
\\
& \qquad = \lambda
x^*_2(t)\frac{\partial}{\partial x^*_2}  +	\lambda x^*_3(t)\frac{\partial}{\partial  x^*_3} + (\lambda+1)x^*_4(t)\frac{\partial}{\partial  x^*_4} + (\lambda+1)x^*_5(t)\frac{\partial}{\partial  x^*_5} \nonumber\\
\Leftrightarrow &  \begin{cases}
x^{*'}_1(t)=x^{*'}(t)= y^{*'}(t)=0
\\
x^{*'}_2(t)=  \lambda
x^*_2(t)
\\
x^{*'}_3(t)=\lambda x^{*}_3(t)
\\
x^{*'}_4(t)=(\lambda+1)x^*_4(t)
\\
x^{*'}_5(t)=(\lambda+1)x^*_5(t).
\end{cases}
\end{align}
Combining with condition $F=\varphi(0)$, we obtain
\begin{equation}\label{X2G3}
x^*_1=\alpha_1,\, x^*_2= \alpha_2e^{\lambda t},\, x^*_3= \alpha_3e^{\lambda t},\, x^*_4=\alpha_4e^{(\lambda+1)t},\, x^*_5=\alpha_5e^{(\lambda+1)t},\, x^*=\alpha,\, y^*=\beta.
\end{equation}
Therefore, the flow of $\mathfrak{X}_2$ is
\[
\theta^{\X_2}_{x}: F \mapsto \theta^{\X_2}_{x}(F):=(\alpha_1,  \alpha_2e^{\lambda x},\alpha_3e^{\lambda x}, \alpha_4e^{(\lambda+1)x}, \alpha_5e^{(\lambda+1)x}, \alpha, \beta).
\]
Next, we consider $\mathfrak{X}_3:= (x^*_2+x^*_3)\frac{\partial}{\partial  x^*_2} + x^*_3\frac{\partial}{\partial  x^*_3} + (x^*_4+x^*_5)\frac{\partial}{\partial  x^*_4} +  x^*_5 \frac{\partial}{\partial  x^*_5}$. For some positive $\epsilon \in {\mathbb R}$, assume~that
\[\varphi: t \mapsto \varphi(t)=\bigl(x^*_1(t), x^*_2(t), x^*_3(t), x^*_4(t), x^*_5(t), x^*(t), y^*(t)\bigr); t \in (-\epsilon, \epsilon)\]
is an integral curve of $\mathfrak{X}_3$ passing $F=\varphi(0)$.
Then, $\varphi'(t)={\X_2}_{\varphi(t)}$ which is equivalent to
\begin{align}\label{G13X3}
& \sum_{i=1}^{5}x^{*'}_i(t)\frac{\partial}{\partial x^*_i}+x^{*'}(t)\frac{\partial}{\partial x^*}+y^{*'}(t)\frac{\partial}{\partial y^*}= \nonumber
\\
& \qquad =\left[ x^*_2(t)+x^*_3(t)\right] \frac{\partial}{\partial  x^*_2} + x^*_3(t)\frac{\partial}{\partial  x^*_3} + \left[ x^*_4(t)+x^*_5(t)\right] \frac{\partial}{\partial  x^*_4} +  x^*_5(t) \frac{\partial}{\partial  x^*_5} \nonumber\\
\Leftrightarrow &  \begin{cases}
x^{*'}_1(t)= x^{*'}(t)= y^{*'}(t)=0
\\
x^{*'}_2(t)=  x^*_2(t)+x^*_3(t)
\\
x^{*'}_3(t)=x^{*}_3(t)
\\
x^{*'}_4(t)=x^*_4(t)+x^*_5(t)
\\
x^{*'}_5(t)= x^*_5(t).
\end{cases}
\end{align}
Combining with condition $F=\varphi(0)$, systemn \eqref{G13X3} gives us
\begin{equation*}\label{X3G13}
\begin{array}{ll}
&x^*_1=\alpha_1,
\; x^*_2= \alpha_2e^{t}+\alpha_3te^{t},
\; x^*_3= \alpha_3e^{t},
\\
&x^*_4=\alpha_4e^{t}+\alpha_5ye^{t},
\; x^*_5=\alpha_5e^{t},
\; x^*=\alpha, \quad y^*=\beta.
\end{array}
\end{equation*}
Hence, the flow of $\mathfrak{X}_3$ is
\[
\theta^{\X_3}_{y}: F \mapsto \theta^{\X_3}_{y}(F):=(\alpha_1, \alpha_2e^{y}+\alpha_3ye^{y},\alpha_3e^{y}, \alpha_4e^{y}+\alpha_5ye^{y}, \alpha_5e^{y}, \alpha, \beta).
\]
Now, we consider $\X_4:=x^*_4 \frac{\partial}{\partial  x^*_2} +x^*_5\frac{\partial}{\partial  x^*_3}$. Assume that
\[\varphi: t \mapsto \varphi(t)=\bigl(x^*_1(t), x^*_2(t), x^*_3(t), x^*_4(t), x^*_5(t), x^*(t), y^*(t)\bigr)\]
is an integral curve of $\X_4$ passing $F=\varphi(0)$, where $ t \in (-\epsilon, \epsilon) \subset \R$ for some positive $\epsilon \in \R$. 
Then we have $\varphi'(t)={\X_4}_{\varphi(t)}$ which means that
\begin{align}\label{G3X4}
& \sum_{i=1}^{5}x^{*'}_i(t)\frac{\partial}{\partial x^*_i}+x^{*'}(t)\frac{\partial}{\partial x^*}+y^{*'}( t)\frac{\partial}{\partial y^*}=x^*_4(t) \frac{\partial}{\partial  x^*_2} +x^*_5( t)\frac{\partial}{\partial  x^*_3} \nonumber\\
\Leftrightarrow &  \begin{cases}
x^{*'}_1(t)= x^{*'}_4(t)=
x^{*'}_5(t)= x^{*'}(t)= y^{*'}(t)=0\\
x^{*'}_2(t)=x^*_4(t)\\
x^{*'}_3(t)=x^*_5(t). 
\end{cases} 
\end{align}
Since $F=\varphi(0)$, system \eqref{G3X4} gives us
\begin{align}\label{X4G3}
&x^*_1=\alpha_1,\;x^*_2= \alpha_2+\alpha_4 t,\; x^*_3=\alpha_3+\alpha_5 t,\;
x^*_4=\alpha_4,	\; x^*_5=\alpha_5,\; x^*=\alpha,\; y^*=\beta.
\end{align}
Therefore, the flow of $\mathfrak{X}_4$ is
\[
\theta^{\X_4}_{x_1}: F \mapsto \theta^{\X_4}_{x_1}(F):=(\alpha_1,\alpha_2+\alpha_4x_1,\alpha_3+\alpha_5x_1,\alpha_4, \alpha_5, \alpha, \beta).
\]
By setting $\theta=\theta^{\mathfrak{X}_6}_{y^*-\beta}\circ\theta^{\mathfrak{X}_5}_{x^*-\alpha}\circ\theta^{\mathfrak{X}_4}_{x_1}\circ\theta^{\mathfrak{X}_3}_{y}\circ\theta^{\mathfrak{X}_2}_{x}\circ\theta^{\mathfrak{X}_1}_{x^*_1-\alpha_1}$, we have
\begin{flalign*}
\theta(F) & =\theta^{\mathfrak{X}_6}_{y^*-\beta}\circ\theta^{\mathfrak{X}_5}_{x^*-\alpha}\circ\theta^{\mathfrak{X}_4}_{x_1}\circ\theta^{\mathfrak{X}_3}_{y}\circ\theta^{\mathfrak{X}_2}_{x}\circ\theta^{\mathfrak{X}_1}_{x^*_1-\alpha_1}(F)
\\
& =(x^*_1, x^*_2, x^*_3, x^*_4, x^*_5, x^*, y^*)
\end{flalign*}
where $x^*_1, x^*, y^* \in \R$ and
\[
\begin{array}{l}
x^*_2= \alpha_2e^{y}+ \alpha_4x_1e^{y}, \;x^*_3= \alpha_3 e^{\lambda_1x+\lambda_2y} +\alpha_5x_1e^{\lambda_1x+\lambda_2y},
\;
x^*_4= \alpha_4e^{x+y},\; x^*_5=\alpha_5 e^{\lambda_1x+x+\lambda_2y}.
\end{array}	
\]
By direct calculations, we get 
\[\begin{array}{ll}
&	x^*_2
= \frac{x^*_3x^*_4}{x^*_5} + \left( \alpha_2-\frac{\alpha_3\alpha_4}{\alpha_5}\right) \frac{{x^*_5}^{\frac{\lambda}{1+\lambda}}e^{\frac{x^*_4}{(1+\lambda)x^*_5}}}{\alpha_5^{\frac{\lambda}{1+\lambda}}e^{\frac{\alpha_4}{(1+\lambda)\alpha_5}}}  
\\
\Leftrightarrow &\left(x^*_2-\frac{x^*_3x^*_4}{x^*_5}\right)\frac{1}{{x^*_5}^{\frac{\lambda}{1+\lambda}}e^{\frac{x^*_4}{(1+\lambda)x^*_5}}}  =\left(\alpha_2-\frac{\alpha_3\alpha_4}{\alpha_5}\right)\frac{1}{\alpha_5^{\frac{\lambda}{1+\lambda}}e^{\frac{\alpha_4}{(1+\lambda)\alpha_5}}} \cdot
\end{array}\]
Hence, 
$\left\lbrace \theta(F) \, | \, x^*_1, x^*_3, x^*_4, x^*_5, x^*, y^* \in \R; \alpha_5x^*_5>0 \right\rbrace \equiv \Omega_F$,
i.e. $\Omega_F$ is a maximal connected integral submanifold corresponding to $S_{G}$. Therefore, $S_{G}$ generates $\mathcal{F}_{G}$ and $(V_{G},\mathcal{F}_{G})$ is a six-dimensional foliation for $G = G_{12}^{\lambda}$. 	

 \noindent {\it Proof of Step 2.} Similar to the case $G=G_{4}^{\lambda_1,\lambda_2}$, we also show that the foliation $(V_{G},\mathcal{F}_{G})$ is measurable in the sense of Connes for $G = G_{12}^{\lambda}$. The proof is complete for the case $G = G_{12}^{\lambda}$.

Finally, let $G = G_{13}^{\lambda}$. {\it Proof of Step 1.} For any $F(\alpha_1,\alpha_2, \alpha_3, \alpha_4, ,\alpha_5,\alpha,\beta) \in \mathcal{G}^*$,
by Theorem \ref{picture K-orbit}, the generic K-orbit $\Omega_F$ belongs to $\mathcal{F}_G$ if and only if $\alpha_4\alpha_5 \neq 0$. 
Furthermore, if we denote by $v$ the element $(x^*_1, x^*_2, x^*_3, x^*_4, x^*_5, x^*, y^*) \in \G^*$ then
\[\begin{array}{r}
\Omega_F=\Bigg\{v \in \mathcal{G}^* \, \bigg| \; \frac{x^*_2x^*_5-x^*_3x^*_4}{{x^{*}_4}^2+{x^{*}_5}^2}e^{\lambda \arctan\frac{x^*_4}{x^*_5}}
=\frac{\alpha_2\alpha_5-\alpha_3\alpha_4}{{\alpha_4}^2+{\alpha_5}^2}e^{\lambda\arctan\frac{\alpha_4}{\alpha_5}} \Bigg\}.\end{array}\]
On the open submanifold $V_G = V_{3}$ as in \eqref{FoliatedManifold2}, we consider the following differential system $S_{G}$:
\[\begin{cases}
\mathfrak{X}_1:=\frac{\partial}{\partial x^*_1}
\\
\mathfrak{X}_2:=  
x^*_2\frac{\partial}{\partial  x^*_2}  +x^*_3\frac{\partial}{\partial  x^*_3} + x^*_4\frac{\partial}{\partial  x^*_4} + x^*_5\frac{\partial}{\partial  x^*_5}
\\
\mathfrak{X}_3:= 
x^*_3\frac{\partial}{\partial  x^*_2}  - x^*_2\frac{\partial}{\partial  x^*_3} + \lambda x^*_5\frac{\partial}{\partial  x^*_4} -\lambda x^*_4\frac{\partial}{\partial  x^*_5}
\\
\mathfrak{X}_4:=x^*_4\frac{\partial}{\partial  x^*_2} +x^*_5\frac{\partial}{\partial  x^*_3}  
\\
\mathfrak{X}_5:= \frac{\partial}{\partial x^*} 
\\
\mathfrak{X}_6:=\frac{\partial}{\partial y^*} \cdot
\end{cases}
\]
Obviously, $\Rank(S_G)=6$ and all $\mathfrak{X}_i \,(i=1, \ldots, 6)$ are smooth over $V_{G}$. 
Moreover, we will show that $S_{G}$ generates $\mathcal{F}_{G}$, i.e. each K-orbit $\Omega$ from $\F_G$ is a 
maximal connected integral submanifold of $S_{G}$.

\noindent First, we consider 
$\mathfrak{X}_1, \mathfrak{X}_5, \mathfrak{X}_6$. Clearly, their flows are determined as follows
\[	\begin{array}{ll}
\theta^{\mathfrak{X}_1}_{x^*_1-\alpha_1}:& F \mapsto \theta^{\mathfrak{X}_1}_{x^*_1-\alpha_1}(F):=(x^*_1, \alpha_2,\alpha_3, \alpha_4, \alpha_5, \alpha, \beta)\\
\theta^{\mathfrak{X}_5}_{x^*-\alpha}:& F \mapsto \theta^{\mathfrak{X}_5}_{x^*-\alpha}(F):=(\alpha_1, \alpha_2,\alpha_3, \alpha_4, \alpha_5, x^*, \beta)\\
\theta^{\mathfrak{X}_6}_{y^*-\beta}:& F \mapsto \theta^{\mathfrak{X}_6}_{y^*-\beta}(F):=(\alpha_1, \alpha_2,\alpha_3, \alpha_4, \alpha_5, \alpha, y^*).
\end{array}
\]
Next, we consider $\mathfrak{X}_2:= x^*_2\frac{\partial}{\partial  x^*_2}  +x^*_3\frac{\partial}{\partial  x^*_3} + x^*_4\frac{\partial}{\partial  x^*_4} + x^*_5\frac{\partial}{\partial  x^*_5}$. For some positive $\epsilon \in {\mathbb R}$, assume that
\[\varphi: t \mapsto \varphi(t)=\bigl(x^*_1(t), x^*_2(t), x^*_3(t), x^*_4(t), x^*_5(t), x^*(t), y^*(t)\bigr); t \in (-\epsilon, \epsilon)\]
is an integral curve of $\mathfrak{X}_2$ passing $F=\varphi(0)$.
Hence, we have $\varphi'(t)={\X_2}_{\varphi(t)}$, and this means that
\begin{align}\label{G13X2}
& \sum_{i=1}^{5}x^{*'}_i(t)\frac{\partial}{\partial x^*_i}+x^{*'}(t)\frac{\partial}{\partial x^*}+y^{*'}(t)\frac{\partial}{\partial y^*} =
x^*_2(t)\frac{\partial}{\partial  x^*_2}  +x^*_3(t)\frac{\partial}{\partial  x^*_3} + x^*_4(t)\frac{\partial}{\partial  x^*_4} + x^*_5(t)\frac{\partial}{\partial  x^*_5} \nonumber\\
\Leftrightarrow &  \begin{cases}
x^{*'}_1(t)= x^{*'}(t)= y^{*'}(t)=0
\\
x^{*'}_2(t)= x^{*}_2(t)
\\
x^{*'}_3(t)=x^{*}_3(t)
\\
x^{*'}_4(t)=x^*_4(t)
\\
x^{*'}_5(t)=x^*_5(t).
\end{cases}
\end{align}
Combining with condition $F=\varphi(0)$, we obtain
\begin{equation*}\label{X2G13}
x^*_1=\alpha_1, \, x^*_2= \alpha_2e^{t},\, x^*_3= \alpha_3e^{t},\, x^*_4=\alpha_4e^t,\, x^*_5=\alpha_5e^{t},\, x^*=\alpha,\, y^*=\beta.
\end{equation*}
Therefore, the flow of $\mathfrak{X}_2$ is
\[
\theta^{\X_2}_{x}: F \mapsto \theta^{\X_2}_{x}(F):=(\alpha_1, \alpha_2,\alpha_3e^{x}, \alpha_4e^x, \alpha_5e^{x}, \alpha, \beta).
\]
Next, we consider $\mathfrak{X}_3:= x^*_3\frac{\partial}{\partial  x^*_2}  - x^*_2\frac{\partial}{\partial  x^*_3} + \lambda x^*_5\frac{\partial}{\partial  x^*_4} -\lambda x^*_4\frac{\partial}{\partial  x^*_5}$. For some positive $\epsilon \in {\mathbb R}$, assume that
\[\varphi: t \mapsto \varphi(t)=\bigl(x^*_1(t), x^*_2(t), x^*_3(t), x^*_4(t), x^*_5(t), x^*(t), y^*(t)\bigr); t \in (-\epsilon, \epsilon)\]
is an integral curve of $\mathfrak{X}_3$ passing $F=\varphi(0)$.
Hence, we have $\varphi'(t)={\X_3}_{\varphi(t)}$, that is
\begin{align}\label{G13X3*}
& \sum_{i=1}^{5}x^{*'}_i(t)\frac{\partial}{\partial x^*_i}+x^{*'}(t)\frac{\partial}{\partial x^*}+y^{*'}(t)\frac{\partial}{\partial y^*}= 
x^*_3(t)\frac{\partial}{\partial  x^*_2}  - x^*_2(t)\frac{\partial}{\partial  x^*_3} + \lambda x^*_5(t)\frac{\partial}{\partial  x^*_4} -\lambda x^*_4(t)\frac{\partial}{\partial  x^*_5} \nonumber\\
\Leftrightarrow &  \begin{cases}
x^{*'}_1(t)= x^{*'}(t)= y^{*'}(t)=0
\\
x^{*'}_2(t)= x^{*}_3(t)
\\
x^{*'}_3(t)=-x^{*}_2(t)
\\
x^{*'}_4(t)=\lambda x^*_5(t)
\\
x^{*'}_5(t)=-\lambda x^*_4(t).
\end{cases}
\end{align}
Combining with condition $F=\varphi(0)$, system \eqref{G13X3*} gives us
\begin{equation*}\label{X3G13*}
\begin{array}{ll}
&x^*_1=\alpha_1,
\\
&x^*_2= \alpha_2\cos{t}+\alpha_3\sin{t} ,
\\
&x^*_3= -\alpha_2\sin{t}+\alpha_3\cos{t},
\\
&x^*_4=(\alpha_4\cos{t}+\alpha_5\sin{t})e^{\lambda t},
\\
& x^*_5=(-\alpha_4\sin{t}+\alpha_5\cos{t})e^{\lambda t},
\\
& x^*=\alpha, \quad y^*=\beta.
\end{array}
\end{equation*}
Hence, the flow of $\mathfrak{X}_3$ is
\[\begin{array}{l}
\theta^{\X_3}_{y}: F \mapsto \theta^{\X_3}_{y}(F):=\big(\alpha_1, \alpha_2\cos{y}+\alpha_3\sin{y},-\alpha_2\sin{y}+\alpha_3\cos{y},\\ \hspace{5cm}(\alpha_4\cos{y}+\alpha_5\sin{y})e^{\lambda y}, (-\alpha_4\sin{y}+\alpha_5\cos{y})e^{\lambda y}, \alpha, \beta\big).
\end{array}	
\]
Now, we consider $\X_4:=x^*_4 \frac{\partial}{\partial  x^*_2} +x^*_5\frac{\partial}{\partial  x^*_3}$. Assume that
\[\varphi: t \mapsto \varphi(t)=\bigl(x^*_1(t), x^*_2(t), x^*_3(t), x^*_4(t), x^*_5(t), x^*(t), y^*(t)\bigr)\]
is an integral curve of $\X_4$ passing $F=\varphi(0)$, where $t \in (-\epsilon, \epsilon) \subset \R$ for some positive $\epsilon \in \R$. 
Then, $\varphi'(t)={\X_4}_{\varphi(t)}$ which is equivalent to
\begin{align}\label{G13X4}
& \sum_{i=1}^{5}x^{*'}_i(t)\frac{\partial}{\partial x^*_i}+x^{*'}(t)\frac{\partial}{\partial x^*}+y^{*'}(t)\frac{\partial}{\partial y^*}=x^*_4(t) \frac{\partial}{\partial  x^*_2} +x^*_5(t)\frac{\partial}{\partial  x^*_3} \nonumber\\
\Leftrightarrow &  \begin{cases}
x^{*'}_1(t)= x^{*'}_4(t)=
x^{*'}_5(t)= x^{*'}(t)= y^{*'}(t)=0\\
x^{*'}_2(t)=x^*_4(t)\\
x^{*'}_3(t)=x^*_5(t). 
\end{cases} 
\end{align}
Since $F=\varphi(0)$, system \eqref{G13X4} gives us
\begin{align*}\label{X4G13}
&x^*_1=\alpha_1,\;x^*_2= \alpha_2+\alpha_4t,\; x^*_3=\alpha_3+\alpha_5t,\;
x^*_4=\alpha_4,	\; x^*_5=\alpha_5,\; x^*=\alpha,\; y^*=\beta.
\end{align*}
Thus, the flow of $\mathfrak{X}_4$ is
\[
\theta^{\X_4}_{x_1}: F \mapsto \theta^{\X_4}_{x_1}(F):=(\alpha_1, \alpha_2+\alpha_4x_1,\alpha_3+\alpha_5x_1, \alpha_4, \alpha_5, \alpha, \beta).
\]
By setting $\theta=\theta^{\mathfrak{X}_6}_{y^*-\beta}\circ\theta^{\mathfrak{X}_5}_{x^*-\alpha}\circ\theta^{\mathfrak{X}_4}_{x_1}\circ\theta^{\mathfrak{X}_3}_{y}\circ\theta^{\mathfrak{X}_2}_{x}\circ\theta^{\mathfrak{X}_1}_{x^*_1-\alpha_1}$, we have
\[
\begin{array}{rl}
\theta(F)=&\theta^{\mathfrak{X}_6}_{y^*-\beta}\circ\theta^{\mathfrak{X}_5}_{x^*-\alpha}\circ\theta^{\mathfrak{X}_4}_{x_1}\circ\theta^{\mathfrak{X}_3}_{y}\circ\theta^{\mathfrak{X}_2}_{x}\circ\theta^{\mathfrak{X}_1}_{x^*_1-\alpha_1}(F)
\\
=&(x^*_1, x^*_2, x^*_3, x^*_4, x^*_5, x^*, y^*).
\end{array}
\]
where $x^*_1, x^*, y^* \in \R$ and
\[
\begin{array}{ll}
x^*_2=\alpha_2e^{x}\cos(y)+\alpha_3e^{x}\sin(y)+\alpha_4x_1e^{x}\cos(y)+\alpha_5x_1e^{x}\sin(y) ,
\\
x^*_3= -\alpha_2e^{x}\sin(y)+\alpha_3e^{x}\cos(y)-\alpha_4x_1e^{x}\sin(y)+\alpha_5x_1e^{x}\cos(y), 
\\
x^*_4=\alpha_4e^{x+\lambda y}\cos(y)+\alpha_5e^{x+\lambda y}\sin(y),
\\
x^*_5=-\alpha_4e^{x+\lambda y}\sin(y)+\alpha_5e^{x+\lambda y}\cos(y) .
\end{array}	
\]
By direct calculations, we get 
\[\begin{array}{ll}
&	x^*_2
= \frac{x^*_3x^*_4}{x^*_5} + \frac{\alpha_2\alpha_5-\alpha_3\alpha_4}{{\alpha_4}^2+{\alpha_5}^2}\frac{e^{\lambda\arctan\frac{\alpha_4}{\alpha_5}}\left( {x^{*}_4}^2+{x^{*}_5}^2\right)  }{x^*_5e^{\lambda \arctan\frac{x^*_4}{x^*_5}}}
\\
\Leftrightarrow &\frac{x^*_2x^*_5-x^*_3x^*_4}{{x^{*}_4}^2+{x^{*}_5}^2}e^{\lambda \arctan\frac{x^*_4}{x^*_5}}
=\frac{\alpha_2\alpha_5-\alpha_3\alpha_4}{{\alpha_4}^2+{\alpha_5}^2}e^{\lambda\arctan\frac{\alpha_4}{\alpha_5}} .
\end{array}\]
Hence, 
$\left\lbrace \theta(F) \, | \, x^*_1, x^*_3, x^*_4, x^*_5, x^*, y^* \in \R; {x^*_4}^2+{x^*_5}^2\neq0 \right\rbrace \equiv\Omega_F$,
i.e. $\Omega_F$ is a maximal connected integral submanifold of $S_{G}$. 
Therefore, $S_{G}$ generates $\mathcal{F}_{G}$ and $(V_{G},\mathcal{F}_{G})$ is a six-dimensional foliation for $G = G_{13}^{\lambda}$.	

\noindent {\it Proof of Step 2.} Similar to the case $G=G_{4}^{\lambda_1,\lambda_2}$, we also show that the foliation $(V_{G},\mathcal{F}_{G})$ is measurable in the sense of Connes for $G = G_{13}^{\lambda}$.

	For the remaining cases, the proof is entirely analogous. 
	Here, we only give the systems $S_G$ in Table \ref{tab3} in which the disappearance of coefficients means they are zeros.
	 
	\begin{longtable}{|p{.07\textwidth} | p{.03\textwidth} | p{.05\textwidth} | p{.12\textwidth} | p{.12\textwidth} |
		 p{.12\textwidth} | p{.12\textwidth} | p{.05\textwidth} | p{.05\textwidth}|}
		\caption{The differential systems $S_G$}\label{tab3} \\
		\hline \multirow{2}{*}{$S_G$} & \multirow{2}{*}{$\X_i$} & \multicolumn{7}{c|}{Coefficients} \\
			\cline{3-9} & & $\frac{\partial}{\partial x^*_1}$ & $\frac{\partial}{\partial x^*_2}$ & $\frac{\partial}{\partial x^*_3}$ 
			& $\frac{\partial}{\partial x^*_4}$ & $\frac{\partial}{\partial x^*_5}$ & $\frac{\partial}{\partial x^*}$ & $\frac{\partial}{\partial y^*}$ \endfirsthead
		\caption*{Table 3 (continued)} \\
		\hline \multirow{2}{*}{$S_G$} & \multirow{2}{*}{$\X_i$} & \multicolumn{7}{c|}{Coefficients} \\
			\cline{3-9} & & $\frac{\partial}{\partial x^*_1}$ & $\frac{\partial}{\partial x^*_2}$ & $\frac{\partial}{\partial x^*_3}$ 
			& $\frac{\partial}{\partial x^*_4}$ & $\frac{\partial}{\partial x^*_5}$ & $\frac{\partial}{\partial x^*}$ & $\frac{\partial}{\partial y^*}$ \\
		\hline \endhead \hline
			 \multirow{6}{*}{$S_{G_1^\lambda}$} & $\X_1$ & 1 &  &  &  &  &  &  \\
			 	& $\X_2$ &  &  & $x_3^*$ &  & $x_5^*$ &  &  \\
			 	& $\X_3$ &  & $-x_2^*$ &  &  & $x_5^*$ &  &  \\
			 	& $\X_4$ &  & $x_4^*$ & $x_5^*$ &  &  &  &  \\
			 	& $\X_5$ &  &  &  &  &  & 1 &  \\
			 	& $\X_6$ &  &  &  &  &  &  & 1 \\ \hline
			\multirow{6}{*}{$S_{G_5}$} & $\X_1$ & 1 &  &  &  &  &  &  \\
			 	& $\X_2$ &  &  & $x_3^*$ &  & $x_5^*$ &  &  \\
			 	& $\X_3$ &  & $x_2^*$ &  & $2x_4^*$ & $x_5^*$ &  &  \\
			 	& $\X_4$ &  & $x_4^*$ & $x_5^*$ &  &  &  &  \\
			 	& $\X_5$ &  &  &  &  &  & 1 &  \\
			 	& $\X_6$ &  &  &  &  &  &  & 1 \\ \hline
			 \multirow{6}{*}{$S_{G_6^\lambda}$} & $\X_1$ & 1 &  &  &  &  &  &  \\
			 	& $\X_2$ &  &  & $x_3^*$ &  & $x_5^*$ &  &  \\
			 	& $\X_3$ &  & $x_2^*$ & $\lambda x_3^*$ & $2x_4^*$ & $(\lambda +1)x_5^*$ &  &  \\
			 	& $\X_4$ &  & $x_4^*$ & $x_5^*$ &  &  &  &  \\
			 	& $\X_5$ &  &  &  &  &  & 1 &  \\
			 	& $\X_6$ &  &  &  &  &  &  & 1 \\ \hline
			 \multirow{6}{*}{$S_{G_7}$} & $\X_1$ & 1 &  &  &  &  &  &  \\
			 	& $\X_2$ &  & $x_2^*$ & $x_3^*$ & $x_4^*$ & $x_5^*$ &  &  \\
			 	& $\X_3$ &  & $x_2^* + x_5^*$ &  & $2x_4^*$ & $x_5^*$ &  &  \\
			 	& $\X_4$ &  & $x_4^*$ & $x_5^*$ &  &  &  &  \\
			 	& $\X_5$ &  &  &  &  &  & 1 &  \\
			 	& $\X_6$ &  &  &  &  &  &  & 1 \\ \hline
			  \multirow{6}{*}{$S_{G_8^\lambda}$} & $\X_1$ & 1 &  &  &  &  &  &  \\
			 	& $\X_2$ &  & $(\lambda +1)x_2^*$ & $\lambda x_3^*$ & $(\lambda +2)x_4^*$ & $(\lambda +1)x_5^*$ &  &  \\
			 	& $\X_3$ &  & $x_2^* + x_5^*$ & $x_3^*$ & $x_4^*$ & $x_5^*$ &  &  \\
			 	& $\X_4$ &  & $x_4^*$ & $x_5^*$ &  &  &  &  \\
			 	& $\X_5$ &  &  &  &  &  & 1 &  \\
			 	& $\X_6$ &  &  &  &  &  &  & 1 \\ \hline
			 \multirow{6}{*}{$S_{G_{11}}$} & $\X_1$ & 1 &  &  &  &  &  &  \\
			 	& $\X_2$ &  & $x_2^*$ & $x_3^*$ & $x_4^*$ & $x_5^*$ &  &  \\
			 	& $\X_3$ &  & $x_3^*$ &  & $x_4^* + x_5^*$ & $x_5^*$ &  &  \\
			 	& $\X_4$ &  & $x_4^*$ & $x_5^*$ &  &  &  &  \\
			 	& $\X_5$ &  &  &  &  &  & 1 &  \\
			 	& $\X_6$ &  &  &  &  &  &  & 1 \\ \hline
			  \multirow{6}{*}{$S_{G_{14}^{\lambda_1\lambda_2}}$} & $\X_1$ & 1 &  &  &  &  &  &  \\
			 	& $\X_2$ &  & $\lambda_1x_2^*$ & $\lambda_1x_3^*$ & $(1+\lambda_1)x_4^*$ & $(1+\lambda_1)x_5^*$ &  &  \\
			 	& $\X_3$ &  & $\lambda_2x_2^*+x_3^*$ & $-x_2^*+\lambda_2x_3^*$  & $\lambda_2x_4^* - x_5^*$ & $-x_4^*+\lambda_2x_5^*$ &  &  \\
			 	& $\X_4$ &  & $x_4^*$ & $x_5^*$ &  &  &  &  \\
			 	& $\X_5$ &  &  &  &  &  & 1 &  \\
			 	& $\X_6$ &  &  &  &  &  &  & 1 \\ \hline
			  \multirow{6}{*}{$S_{G_{15}}$} & $\X_1$ & 1 &  &  &  &  &  &  \\
			 	& $\X_2$ &  & $x_3^*$ & $-x_2^*$ & $x_5^*$ & $-x_4^*$ &  &  \\
			 	& $\X_3$ &  & $x_2^* + x_5^*$ & $x_3^* - x_4^*$ & $x_4^*$ & $x_5^*$ &  &  \\
			 	& $\X_4$ &  & $x_4^*$ & $x_5^*$ &  &  &  &  \\
			 	& $\X_5$ &  &  &  &  &  & 1 &  \\
			 	& $\X_6$ &  &  &  &  &  &  & 1 \\ \hline
			  \multirow{6}{*}{$S_{G_{16}^\lambda}$} & $\X_1$ & 1 &  &  &  &  &  &  \\
			 	& $\X_2$ &  & $x_3^* + x_5^*$ & $-x_2^*$ & $x_5^*$ & $-x_4^*$ &  &  \\
			 	& $\X_3$ &  & $x_2^* + \lambda x_5^*$ & $x_3^* - \lambda x_4^*$ & $x_4^*$ & $x_5^*$ &  &  \\
			 	& $\X_4$ &  & $x_4^*$ & $x_5^*$ &  &  &  &  \\
			 	& $\X_5$ &  &  &  &  &  & 1 &  \\
			 	& $\X_6$ &  &  &  &  &  &  & 1 \\ \hline
	\end{longtable}	
	The proof of Theorem \ref{FormedFoliation} is complete.
\end{proof}

\begin{defn}
	For each $G \in \A$, the foliation $(V_G, \F_G)$ is called the \emph{generalized MD-foliation} (for short, \emph{GMD-foliation}) associated to $G$.
\end{defn}

\begin{rem}
	By Remark \ref{RemarkGenericOrbits} and Theorem \ref{FormedFoliation}, we have exactly sixteen families of measurable GMD-foliations $\{(V_G, \F_G) / G \in \A\}$. 
	However, there are only three types of foliated manifolds $V_G$. Namely, we~have 	
	\begin{itemize}
		\item For all $G \in \A_1 \setminus \{G_{12}^{\lambda}\}$, $V_G \equiv V_1$ as in \eqref{FoliatedManifold};		
		\item For $G = G_{12}^{\lambda}$, $V_G \equiv V_{2}$ as in \eqref{FoliatedManifold1};  		
		\item For all $G \in \A_2 \big\}$, $V_G \equiv V_{3}$ as in \eqref{FoliatedManifold2}.
		
	\end{itemize}
\end{rem}

It is this observation that leads us to the following result about
the topological classification of sixteen families of GMD-foliations. Namely, we have the following theorem.

\begin{thm}\label{TopoClass}
	The topology of GMD-foliations has the following properties.
	\begin{enumerate}[1.]
		\item \label{TopoTypes} There exist exactly three topological types of sixteen families of considered GMD-foliations as follows:
		\begin{enumerate}
			\item \label{TopoTypes1} {\it The Type $\F_1$}: This type contains eleven families of foliated manifolds from $\big\{(V_G, \F_G) / G \in \A_1 \setminus \{G_{12}^{\lambda}\} \big\}$ on the same foliated manifold $V_G \equiv V_{1}$ as in \eqref{FoliatedManifold}; 
			\item \label{TopoTypes2} {\it The Type $\F_2$}: This type contains only one familiy of foliated manifolds from $\big\{\big(V_2, \F_{G_{12}^\lambda}\big)\big\}$ on the same foliated manifold $V_G \equiv V_{2}$ as in \eqref{FoliatedManifold1}; 
			\item \label{TopoTypes3} {\it The Type $\F_3$}: This type contains four families of foliated manifolds from $\big\{(V_G, \F_G) / G \in \A_2 \big\}$ on the same foliated manifold $V_G \equiv V_{3}$ as in \eqref{FoliatedManifold2}.
			\end{enumerate}	
		\item \label{Type} Furthermore, we have:
		\begin{enumerate}
			\item \label{Type1} All GMD-foliations of type $\F_1$ are trivial fibrations with connected fibers on the disjoint union of four copies of the real line $\R$.
			
			\item \label{Type2} The GMD-foliation of type $\F_2$ is a trivial fibration with connected fibers on the disjoint union of two copies of the real line $\R$.
			
			\item \label{Type3} All GMD-foliations of type $\F_3$ are trivial fibrations with connected fibers on a copy of the real line $\R$.
		\end{enumerate}
	\end{enumerate}
\end{thm}

\begin{proof} 	
	Recall that two foliations $\F$ and $\F'$ on the same foliated manifold $V$ are \emph{topologically equivalent} 
	if there exists a homeomorphism $h \colon V \to V$ which sends leaves of $\F$ onto those of $\F'$.
	\begin{enumerate}
		\item Note that the assertion in Item \eqref{TopoTypes1} of Theorem \ref{TopoClass} for cases $G_2$, $G_3$, $G_9$ and $G^\lambda_{10}$ had been proved in \cite{t-v}. The proof is quite analogous for remaining cases. First, we prove the topological classification of considered GMD-foliations in the first assertion of Theorem \ref{TopoClass}.
		\begin{enumerate}
			\item Consider maps $h_{1(\lambda)}, h_{2(\lambda_1\lambda_2)}, h_{3(\lambda)}, h_{4}, h_{5}, h_{6}$ from $V_{1}$ to $V_{1}$ 
			which are defined as follows:
			\[
				\begin{array}{l}
					h_{1(\lambda)}(v) := \left( x^*_1, x^*_2 , x^*_3,  x^*_2-\frac{x^*_3 x^*_4}{x^*_5}, x^*_5, x^*, y^*\right), \\
					h_{2(\lambda_1\lambda_2)}(v) := \left( x^*_1, 
						\frac{x^*_2{x^*_5}^{\frac{1}{\lambda_2-\lambda_1-1}}}{{x^*_4}^{\frac{1+\lambda_1}{\lambda_2-\lambda_1-1}}}, 
						\frac{x^*_3{x^*_5}^{\frac{1}{\lambda_2-\lambda_1-1}}}{{x^*_4}^{\frac{1+\lambda_1}{\lambda_2-\lambda_1-1}}}, 
						x^*_4, x^*_5, x^*, y^*\right), \\
					h_3(v) := \left( x^*_1, \left( x^*_2 +  \ln \vert x^*_4 \vert\right) x^*_5, 
						\left( x^*_3+\frac{x^*_5 \ln \vert x^*_5 \vert}{x^*_4}\right) x^*_5, x^*_4, x^*_5, x^*, y^*\right), \\
					h_{4(\lambda)}(v) := \left( x^*_1, \left( x^*_2 - (1+\lambda) \ln \vert x^*_4 \vert\right) x^*_5, 
						\left( x^*_3 - \frac{(2+\lambda)x^*_5 \ln \vert x^*_5 \vert}{x^*_4}\right) x^*_5, x^*_4, x^*_5, x^*, y^*\right), \\
					h_5(v) := \left( x^*_1,\frac{x^*_2x^*_5}{e^{\frac{x^*_4}{x_5}}} , \frac{x^*_3x^*_5}{e^{\frac{x^*_4}{x_5}}}, x^*_4, x^*_5, x^*, y^*\right), \\
					h_6(v) := \left( x^*_1,\frac{x^*_2}{\sqrt{\vert x^*_4\vert}} , \frac{x^*_3}{\sqrt{\vert x^*_4\vert}}, x^*_4, x^*_5, x^*, y^*\right) 
				\end{array}
			\]
			where $v: = (x^*_1,x^*_2, x^*_3, x^*_4, x^*_5, x^*, y^*) \in V_1$. It is clear that all of above maps are homeomorphisms.
			Now, we take an arbitrary leaf $L$ of $(V_{1}, \F_{G_2})$. Without loss of generality, assume that $L \subset {V_{1}}_{++} \subset V_{1}$
			(see \eqref{SS1} and \eqref{SS2}), i.e. $L$ is determined as follows
			\[
				L  = \left\lbrace v \in V_{1} \, \Big| \, x^*_2-\frac{x^*_3 x^*_4}{x^*_5} = c; \, \, x^*_4 > 0, \, x^*_5 > 0\right\rbrace
			\]
			where $c \in \R$ is some constant. For $\left( V_{1}, \mathcal{F}_{G_4^{\lambda_1\lambda_2}}\right) $, we consider the leaf 
			$\tilde{L} \subset {V_1}_{++} \subset V_1$ which is determined as follows
			\[
				\tilde{L} = \left\lbrace \tilde{v} \in V_1 \, \Big| \left( \tilde{x}_2-\frac{\tilde{x}_3\tilde{x}_4}{\tilde{x}_5}\right)
				\dfrac{\tilde{x}_4^{\frac{1+\lambda_1}{\lambda_2-\lambda_1-1}}}{{\tilde{x}_5}^{\frac{1}{\lambda_2-\lambda_1-1}}}=c; 
				\; \widetilde{x}_4>0,\, \widetilde{x}_5>0\right\rbrace
			\]
			where $\tilde{v} = (\tilde{x}_1, \tilde{x}_2, \tilde{x}_3, \tilde{x}_4, \tilde{x}_5, \tilde{x}, \tilde{y}) \in V_{1}$. 
			By the formula of the homeomorphism $h_{2(\lambda_1\lambda_2)}$, the fact that $h_{2(\lambda_1,\lambda_2)}(v) = \tilde{v}$ is equivalent to
			\[
				\begin{array}{rl}
					& \tilde{v} = \left( x^*_1,\frac{x^*_2{x^*_5}^{\frac{1}{\lambda_2-\lambda_1-1}}}{{x^*_4}^{\frac{1+\lambda_1}{\lambda_2-\lambda_1-1}}}, 
						\frac{x^*_3{x^*_5}^{\frac{1}{\lambda_2-\lambda_1-1}}}{{x^*_4}^{\frac{1+\lambda_1}{\lambda_2-\lambda_1-1}}}, x^*_4, x^*_5, x^*, y^*\right) \\		
					\Leftrightarrow &
						\begin{cases}
							\tilde{x}_1=x^*_1 \\
							\tilde{x}_2 = \frac{x^*_2{x^*_5}^{\frac{1}{\lambda_2-\lambda_1-1}}}{{x^*_4}^{\frac{1+\lambda_1}{\lambda_2-\lambda_1-1}}} \\
							\tilde{x}_3 = \frac{x^*_3{x^*_5}^{\frac{1}{\lambda_2-\lambda_1-1}}}{{x^*_4}^{\frac{1+\lambda_1}{\lambda_2-\lambda_1-1}}} \\
							\tilde{x}_4=x^*_4, \, \tilde{x}_5=x^*_5, \, \tilde{x}=x^*, \, \tilde{y}=y^*.
						\end{cases}
				\end{array}
			\]		
			Therefore
			\[
				\begin{array}{l l l l}
					v \in L & \Leftrightarrow & x^*_2-\frac{x^*_3 x^*_4}{x^*_5} = c; & x^*_4 > 0, \, x^*_5 > 0 \\		
					& \Leftrightarrow & \left( \tilde{x}_2-\frac{\tilde{x}_3\tilde{x}_4}{\tilde{x}_5}\right)
						\frac{\tilde{x}_4^{\frac{1+\lambda_1}{\lambda_2-\lambda_1-1}}}{{\widetilde{x}_5}^{\frac{1}{\lambda_2-\lambda_1-1}}}	= c;
						& \tilde{x}_4>0,\, \tilde{x}_5>0 \\
					& \Leftrightarrow & h_{2(\lambda_1\lambda_2)}(v) = \tilde{v} \in \tilde{L}
				\end{array}
			\]
			i.e. $h_{2(\lambda_1\lambda_2)}(L) = \tilde{L}$ for $L \subset {V_{1}}_{++}$. 
			Similarly, $h_{2(\lambda_1\lambda_2)}(L) = \tilde{L}$ for $L$ in ${V_{1}}_{-+}$, ${V_{1}}_{--}$, ${V_{1}}_{+-}$.
			This means that $h_{2(\lambda_1\lambda_2)}$ sends leaves of $\F_{G_2}$ onto those of $\F_{G_4^{\lambda_1\lambda_2}}$. 
			Therefore, $\left( V_1, \F_{G_2}\right)$ and $\left( V_1, \F_{G_4^{\lambda_1\lambda_2}}\right)$ are topological equivalent.
		
			In the same manner, we can verify that $h_{1(\lambda)}$, $h_3$, $h_{4(\lambda)}$ and $h_5$ send leaves of $\F_{G_2}$ 
			onto those of $\F_{G^\lambda_1}$, $\F_{G_7}$, $\F_{G_8^\lambda}$, $\F_{G_{11}}$, respectively.
			Besides, $h_6$ sends leaves of $\F_{G_3}$ onto those of $\F_{G_5}$ and $\F_{G_6^\lambda}$.  
			Finally, combining with \cite[Theorem 2.3]{t-v}, all these foliations are topologically equivalent. This type is denoted by $\F_1$. 
			
			\item By direct calculations, the map $h_{7(\lambda\neq-1)} \colon V_2 \cong \R^4 \times \left(\R \setminus \{0\}\right) \times \R^2 \to V_2$ 
			defined by
			\[
				h_{7(\lambda\neq-1)}(v) := \left(x^*_1, \frac{x^*_2{x^*_5}^{\frac{\lambda}{1+\lambda}}}{e^{\frac{\lambda x^*_4}{(1+\lambda)x^*_5}}}, 
					\frac{x^*_3{x^*_5}^{\frac{\lambda}{1+\lambda}}}{e^{\frac{\lambda x^*_4}{(1+\lambda)x^*_5}}},x^*_4,x^*_5,x^*,y^* \right)
			\]
			is a homeomorphism which sends leaves of $\F_{G_{12}^0}$ onto those of $\F_{G_{12}^\lambda}$.
			
			In fact, take an arbitrary leaf $L$ of $\F_{G_{12}^0}$. Without loss of generality, assume that $L \subset {V_{2}}_{+} \subset V_{2}$ 
			as in \eqref{FoliatedManifold2} and \eqref{SS3}, i.e. $L$ is determined as follows
			\[
				L  = \left\lbrace v \in V_2 \, \Big| \left(x^*_2-\frac{x^*_3x^*_4}{x^*_5}\right)\frac{1}{e^{\frac{x^*_4}{x^*_5}}}  = c; \; x^*_5 > 0 \right\rbrace
			\]
			where $c \in \R$ is some constant. For $\F_{G_{12}^\lambda}$, we consider the leaf $\tilde{L} \subset {V_{2}}_{+} \subset V_{2}$ 
			as follows
			\[
				\tilde{L} = \left\lbrace  \tilde{v} \in V_{2} \, \Big| \left(\tilde{x}_2-\frac{\tilde{x}_3\tilde{x}_4}{\tilde{x}_5}\right) 
				\frac{1}{{\tilde {x}_5}^{\frac{\lambda}{1+\lambda}}e^{\frac{{\tilde{x}}_4}{(1+\lambda)\tilde{x}_5}}}= c; \, \, \, \tilde{x}_5>0 \right\rbrace
			\]
			where $\tilde{v} = (\tilde{x}_1, \tilde{x}_2, \tilde{x}_3, \tilde{x}_4, \tilde{x}_5, \tilde{x}, \tilde{y}) \in V_2$. 
			By the formula of $h_{7(\lambda\neq-1)}$, it is plain that $h_{7(\lambda\neq-1)}(v) = \tilde{v}$ is equivalent to
			\[
				\begin{array}{rl}
					& \tilde{v} = \left( x^*_1,\frac{x^*_2{x^*_5}^{\frac{1}{\lambda_2-\lambda_1-1}}}{{x^*_4}^{\frac{1+\lambda_1}{\lambda_2-\lambda_1-1}}}, 
						\frac{x^*_3{x^*_5}^{\frac{1}{\lambda_2-\lambda_1-1}}}{{x^*_4}^{\frac{1+\lambda_1}{\lambda_2-\lambda_1-1}}}, x^*_4, x^*_5, x^*, y^*\right) \\		
					\Leftrightarrow & 
						\begin{cases}
							\tilde{x}_1=x^*_1 \\
							\tilde{x}_2=  \frac{x^*_2{x^*_5}^{\frac{\lambda}{1+\lambda}}}{e^{\frac{\lambda x^*_4}{(1+\lambda)x^*_5}}} \\
							\tilde{x}_3 = \frac{x^*_3{x^*_5}^{\frac{\lambda}{1+\lambda}}}{e^{\frac{\lambda x^*_4}{(1+\lambda)x^*_5}}} \\
							\tilde{x}_4=x^*_4, \, \tilde{x}_5=x^*_5, \, \tilde{x}=x^*, \, \tilde{y}=y^*.
						\end{cases}
				\end{array}
			\]		
			Therefore
			\[
				\begin{array}{l l l l}
					v \in L & \Leftrightarrow & \left(x^*_2-\frac{x^*_3x^*_4}{x^*_5}\right)\frac{1}{e^{\frac{x^*_4}{x^*_5}}} = c; & x^*_5 > 0 \\		
					& \Leftrightarrow & \left(\tilde{x}_2-\frac{\tilde{x}_3\tilde{x}_4}{\tilde{x}_5}\right) 
						\frac{1}{{\tilde {x}_5}^{\frac{\lambda}{1+\lambda}}e^{\frac{{\tilde{x}}_4}{(1+\lambda)\tilde{x}_5}}} =c; & \tilde{x}_5>0 \\
					& \Leftrightarrow & h_{7(\lambda\neq-1)}(v) = \tilde{v} \in \tilde{L}
				\end{array}
			\]
			i.e. $h_{7(\lambda\neq-1)}(L) 
			= \tilde{L}$ for $L \subset {V_{2}}_{+}$. Similarly, $h_{7(\lambda\neq-1)}(L) = \tilde{L}$ for $L$ in ${V_{2}}_{-}$.
			Since $h_{7(\lambda\neq-1)}$ sends leaves of $\F_{G_{12}^{0}}$ onto those of $\F_{G_{12}^\lambda}$, 
			all foliations $\left( V_{2}, \F_{G_{12}^{\lambda}}\right)$ are topologically equivelant. This types is denoted by $\F_2$. 
			
			\item Similarly, consider homeomorphisms of $V_3$ as follows:
			\[
				\begin{array}{l}
					h_{8(\lambda)}(v) := 
						\begin{cases}
							\left( x^*_1, \frac{x^*_2}{e^{\lambda a}}, \frac{x^*_3}{e^{\lambda a}}, x^*_4, x^*_5, x^*, y^*\right), 
								& x^*_4x^*_5\neq0 \\
							(x^*_1, x^*_2, x^*_3, 0, x^*_5, x^*, y^*), & x^*_4=0, x^*_5\neq0\\
							(x^*_1, x^*_2, x^*_3, x^*_4, 0, x^*, y^*), & x^*_4\neq0, x^*_5=0
						\end{cases}	\\
					h_{9(\lambda_1\neq-1,\lambda_2)}(v) := 
						\begin{cases}
							\left( x^*_1, \frac{x^*_2}{ce^{\lambda_2 b}}, 
							\frac{x^*_3}{ce^{\lambda b}}, x^*_4, x^*_5, x^*, y^*\right), & x^*_4x^*_5\neq0 \\
							\left( x^*_1, \frac{x^*_2}{c}, x^*_3, 0, x^*_5, x^*, y^*\right), & x^*_4=0 \neq x^*_5 \\
							\left( x^*_1, x^*_2, \frac{x^*_3}{c}, x^*_4, 0, x^*, y^*\right), & x^*_4 \neq 0 = x^*_5 
						\end{cases}	\\
					h_{10}(v) := 
						\begin{cases}
							\left( x^*_1, x^*_2+\frac{d}{2x^*_5}, x^*_3, x^*_4, x^*_5, x^*, y^*\right), & x^*_4x^*_5\neq0 \\
							\left( x^*_1, x^*_2+x^*_5\ln\vert x^*_5\vert, x^*_3, 0, x^*_5, x^*, y^*\right), & x^*_4=0 \neq x^*_5 \\
							\left( x^*_1, x^*_2, x^*_3-x^*_4\ln\vert x^*_4\vert, x^*_4, 0, x^*, y^*\right), & x^*_4 \neq 0 = x^*_5
						\end{cases}	\\
					h_{11(\lambda)}(v) :=
						\begin{cases}
							\left( x^*_1, \frac{2x^*_2 + x^*_4 + b}{2} + \frac{\lambda d}{2x^*_4}, 
							x^*_3, x^* _4, x^*_5, x^*, y^*\right), & x^*_4x^*_5\neq0 \\
							\left( x^*_1, x^*_2+\lambda x^*_5\ln\vert x^*_5\vert, x^*_3, 0, x^*_5, x^*, y^*\right), & x^*_4=0 \neq x^*_5 \\
							\left( x^*_1, x^*_2, x^*_3-\lambda x^*_4\ln\vert x^*_4\vert, x^*_4, 0, x^*, y^*\right), & x^*_4 \neq 0 = x^*_5
						\end{cases}
				\end{array}
			\]
			where $a = \arctan\frac{x^*_4}{x^*_5}$, $b = \arctan\frac{x^*_5}{x^*_4}$, $c = \left({x_4^*}^2+{x_5^*}^2\right)^\frac{1}{1+\lambda_1}$,
			and $d = \left({x^*_4}^2+{x^*_5}^2\right) \ln \left({x^*_4}^2+{x^*_5}^2\right)$. We can check that 
			$h_{8(\lambda)}$, $h_{9(\lambda_1\neq-1,\lambda_2)}$, $h_{10}$ and $h_{11(\lambda)}$ send leaves of $\F_{G_{13}^0}$ 
			onto those of $\F_{G_{13}^\lambda}$, $\F_{G_{14}^{\lambda_1\lambda_2}}$, $\F_{G_{15}}$, $\F_{G_{16}^\lambda}$, respectively. 
			In fact, we will proof the first case in detail, the remaining cases are absolutely similar. 
			First, take an arbitrary leaf $L$ of $\F_{G_{13}^0}$. Without loss of generality, assume that $L \subset V_3$ as in \eqref{FoliatedManifold2}, i.e.
			\[
				L  = \left\lbrace v \in V_3 \; \Big| \; \frac{x^*_2x^*_5-x^*_3x^*_4}{{x^{*}_4}^2+{x^{*}_5}^2}= c; \, \, \tilde{x}^2_4+ \tilde{x}^2_5 \neq 0\right\rbrace
			\]
			where $c \in \R$ is some constant. For $\F_{G_{13}^\lambda}$, consider the leaf $\tilde{L} \subset V_3$ as follows
			\[
				\tilde{L} = \left\lbrace \tilde{v} \in V_3 \; \Big| \; \frac{\tilde{x}_2\tilde{x}_5-\tilde{x}_3\tilde{x}_4}{{\tilde{x}_4}^2+{\tilde{x}_5}^2} 
				e^{\lambda \arctan\frac{\tilde{x}_4}{\tilde{x}_5}}	 =c; \, \, \tilde{x}^2_4 +\tilde{x}^2_5\neq0\right\rbrace
			\]
			where $\tilde{v} = (\tilde{x}_1, \tilde{x}_2, \tilde{x}_3, \tilde{x}_4, \tilde{x}_5, \tilde{x}, \tilde{y}) \in V_3$. 
			By the formula of $h_{8(\lambda)}$, it is plain that $h_{8(\lambda)}(v) = \tilde{v}$ is equivalent to
			\[
				\begin{array}{rl}
					& \tilde{v} =
						\begin{cases}
							\left( x^*_1, \frac{x^*_2}{e^{\lambda a}}, \frac{x^*_3}{e^{\lambda a}}, x^* _4, x^*_5, x^*, y^*\right), 
								&  x^*_4x^*_5 \neq 0 \\
							(x^*_1, x^*_2, x^*_3, 0, x^*_5, x^*, y^*), & x^*_4=0 \neq x^*_5 \\
							(x^*_1, x^*_2, x^*_3, x^*_4, 0, x^*, y^*), & x^*_4 \neq 0 = x^*_5
						\end{cases} \\
					\Leftrightarrow & 
					\begin{cases}
						\tilde{x}_1=x^*_1, \tilde{x}_2=  \frac{x^*_2}{e^{\lambda a}}, \tilde{x}_3 = \frac{x^*_3}{e^{\lambda a}},
						\tilde{x}_4=x^*_4\neq0, \tilde{x}_5=x^*_5\neq0, \tilde{x}=x^*, \tilde{y}=y^* \\
						\tilde{x}_1=x^*_1, \tilde{x}_2= x^*_2, \tilde{x}_3 = x^*_3, \tilde{x}_4=x^*_4=0, \tilde{x}_5=x^*_5 \neq 0, \tilde{x}=x^*, \tilde{y}=y^*\\
						\tilde{x}_1=x^*_1, \tilde{x}_2= x^*_2, \tilde{x}_3 = x^*_3, \tilde{x}_4=x^*_4\neq0, \tilde{x}_5=x^*_5 = 0, \tilde{x}=x^*, \tilde{y}=y^*.
					\end{cases}
				\end{array}
			\]		
			Therefore
			\[
				\begin{array}{l l l l}
					v \in L & \Leftrightarrow & \dfrac{x^*_2x^*_5-x^*_3x^*_4}{{x^{*}_4}^2+{x^{*}_5}^2} = c; & {x^*_4}^2 + {x^*_5}^2 \neq 0 \\		
					& \Leftrightarrow & \dfrac{\tilde{x}_2\tilde{x}_5-\tilde{x}_3\tilde{x}_4}{{\tilde{x}_4}^2
					+ {\tilde{x}_5}^2}e^{\lambda \arctan\frac{\tilde{x}_4}{\tilde{x}_5}}	 =c; & \tilde{x}^2_4 + \tilde{x}^2_5 \neq 0 \\
					& \Leftrightarrow & h_{8(\lambda)}(v) = \tilde{v} \in \tilde{L}
				\end{array}
			\]
			i.e. $h_{8(\lambda)}(L) = \tilde{L}$ for $L \subset V_3$. Therefore, $h_{8(\lambda)}$ sends leaves of $\F_{G_{13}^0}$ onto those of 
			$\F_{G_{13}^\lambda}$, and all foliations $\left( V_3, \F_{G_{13}^\lambda}\right)$ are topological equivalent.
			
				To sum up, all foliations from the set in $\big\{(V_G, \F_G) / G \in \A_2 \big\}$ are topologically equivalent and they determine the type $\F_3$. Thus, the assertion in Item \eqref{TopoTypes3} of Theorem \ref{TopoClass} is proved.	
			Moreover, the non-equivalence of types $\F_1$, $\F_2$ and $\F_3$ are evident since the foliated manifolds $V_1$, $V_2$ and
			$V_3$ are not homeomorphic.
		\end{enumerate}

		\item Now we prove the second of Theorem \ref{TopoClass}.
		\begin{enumerate}
			\item Evidently, 
			by \cite[Theorem 23]{t-v} the assertion in Item \eqref{TopoTypes1} of Theorem \ref{TopoClass} is proved.
			
			\item Due to Item \eqref{TopoTypes2}, we only need to prove this assertion for $\left( V_2, \F_{G_{12}^0}\right)$.
			
			Recall that a subset $W \subset V$ of a foliation $(V, \F)$ is said to be saturated (with respect to the foliation)
			if every leaf $L$ of $\F$ has a non-empty intersection with $W$ then it lies entirely in $W$, i.e. if $L \cap W \neq \emptyset$ then $L \subset W$. 
			If $W \subset V$ is a saturated submanifold then the family of all leaves $L$ of $(V, \F)$ such that $L \subset W$ forms a new foliation on $W$ 
			which is denoted by $(W, \F_W)$ and it is called the \emph{restriction} or \emph{subfoliation} of $(V, \F)$ on $W$.
			
			By \eqref{FoliatedManifold1}, the foliated manifold $V_2$ of all GMD-foliations as the disjoint union of two open subsets 
			${V_2}_{+}$ and ${V_2}_{-}$ given by \eqref{SS3}. These subsets are two connected components of $V_2$ 
			which are saturated with respect to all GMD-foliations.
			
			Now, we consider the $\left( V_2, \F_{G_{12}^0}\right)$. For convenience, we denote its restrictions on ${V_2}_{+}$ and ${V_2}_{-}$ 
			by ${\F_2}_{+}$ and ${\F_2}_{-}$, respectively. Let $p \colon V_2 \to \R$ be the map that defines as follows
			\[
				p(x^*_1,x^*_2, x^*_3, x^*_4, x^*_5, x^*, y^*) := \left(x^*_2-\frac{x^*_3x^*_4}{x^*_5}\right)\frac{1}{e^{\frac{x^*_4}{x^*_5}}}
			\]
			where $(x^*_1,x^*_2, x^*_3, x^*_4, x^*_5, x^*, y^*) \in V_2$. It can be verified that $p_{+} := p\vert _{{V_2}_{+}}$ is a submersion 
			and $p_{+} \colon {V_2}_{+} \to \R$ is a fibration over $\R$ with connected fibers. 
			Moreover, ${\F_2}_{+}$ comes from this fibration. Similarly, the foliation ${\F_2}_{-}$ also comes from a fibration on $\R$ with connected fibers. 
			Therefore, $\left( V_2, \F_{G_{12}^0}\right)$ comes from a fibration with connected fibers on $\R \sqcup \R$, and so does the type $\F_2$. 
			
			\item For the assertion in Item \eqref{TopoTypes3} of Theorem \ref{TopoClass}, we only need to prove for $\left( V_3, \F_{G_{13}^0}\right)$.
			
			
			Let $p \colon V_3 \to \R$ be the map that defines as follows
			\[
			p(x^*_1,x^*_2, x^*_3, x^*_4, x^*_5, x^*, y^*) := \frac{x^*_2x^*_5-x^*_3x^*_4}{{x^{*}_4}^2+{x^{*}_5}^2}
			\]
			where $(x^*_1,x^*_2, x^*_3, x^*_4, x^*_5, x^*, y^*) \in V_3$. It can be verified that $p$ is a submersion 
			and $p$ is a fibration over $\R$ with connected fibers. 
			Therefore, $\left( V_3, \F_{G_{13}^0}\right)$ comes from a fibration with connected fibers on $\R$, and so does the type $\F_3$.		
		\end{enumerate}
	\end{enumerate}
	The proof of Theorem \ref{TopoClass} is complete.
\end{proof}

As an immediate consequence of Theorem \ref{TopoClass} and Connes \cite[Section 5, p. 16]{con}, we have the following.

\begin{cor}\label{Lastcor}
	All Connes' C*-algebras of GMD-foliations in Theorem \ref{TopoClass} are determined as follows:
	\begin{enumerate}
		\item $C^*(\F_1) \cong \left(C_0(\R) \oplus C_0(\R) \oplus C_0(\R) \oplus C_0(\R) \right) \otimes \mathcal{K}$,
		
		\item $C^*(\F_2) \cong \left(C_0(\R) \oplus C_0(\R) \right) \otimes \mathcal{K}$,
		
		\item $C^*(\F_3) \cong  C_0(\R) \otimes \mathcal{K}$,
	\end{enumerate}
	where $C_0(\R)$ is the C*-algebra of continuous complex-valued functions defined on $\R$ vanishing at infinity, 
	and $\mathcal{K}$ denotes the C*-algebra of compact operators on an infinite-dimensional separable Hilbert space.
\end{cor}

\section{Conclusion}

We have considered connected and simply connected Lie groups corresponding to 7-dimensional Lie algebras with nilradical $\g_{5,2}$. 
The main results of the paper are as follows. First, we describe the picture of maximal-dimensional K-orbits of considered Lie groups 
as well as their geometric characteristics in Theorem \ref{picture K-orbit}, Remark \ref{RemarkOrbits} and Remark \ref{RemarkGenericOrbits}.
Afterwards, Theorem \ref{FormedFoliation} proves that the families of all generic maximal-dimensional K-orbits of considered Lie groups 
form measurable foliations (in the sense of Connes) which is called GMD-foliations.
Finally, the topological classification of all GMD-foliations is given Theorem \ref{TopoClass}, and their Connes' C*-algebras
are also described in Corollary \ref{Lastcor}.

\section*{Acknowledgements} 
Tuyen T. M. Nguyen was funded by Vingroup Joint Stock Company and supported by the Domestic Master/PhD Scholarship 
Programme of Vingroup Innovation Foundation (VINIF), Vingroup Big Data Institute (VINBIGDATA), VINIF.2021.TS.159.
A part of this paper was done when Vu A. Le visited Vietnam Institute for Advanced Study in Mathematics (VIASM) in summer 2022,
and the author is very grateful to VIASM for the support and hospitality.


\end{document}